\documentclass[11pt,twoside, leqno]{article}

\usepackage{amsfonts}
\usepackage{mathrsfs}

\usepackage{amssymb}
\usepackage{amsmath}
\usepackage{amsthm}

\usepackage{color}
\usepackage[show]{ed}

\allowdisplaybreaks

\def\rr{{\mathbb R}}

\def\rn{{{\rr}^n}}
\def\zz{{\mathbb Z}}
\def\cc{{\mathbb C}}
\def\cm{{\mathcal M}}

\def\nn{{\mathbb N}}

\def\cx{{\mathcal X}}

\def\fz{\infty}
\def\az{\alpha}
\def\supp{{\mathop\mathrm{\,supp\,}}}
\def\dist{{\mathop\mathrm{\,dist\,}}}
\def\loc{{\mathop\mathrm{\,loc\,}}}

\def\lz{\lambda}
\def\dz{\delta}

\def\ez{\epsilon}

\def\kz{\kappa}
\def\bz{\beta}
\def\fai{\varphi}
\def\gz{{\gamma}}

\def\boz{{\Omega}}

\def\sz{\sigma}

\def\wz{\widetilde}
\def\noz{\nonumber}

\def\hs{\hspace{0.3cm}}

\def\ls{\lesssim}
\def\gs{\gtrsim}

\def\lo{{L^1(\mu)}}
\def\wlo{{L^{1,\,\fz}(\mu)}}
\def\lt{{L^2(\mu)}}

\def\lq{{L^q(\mu)}}
\def\lp{{L^p(\mu)}}
\def\wlp{{L^{p,\,\fz}(\mu)}}

\newcommand{\pair}[2]{\langle #1,#2 \rangle}
\newcommand{\Norm}[2]{\|#1\|_{#2}}
\newcommand{\prob}[0]{\mathbb{P}}
\newcommand{\eps}[0]{\varepsilon}
\newcommand\Z{\mathbb{Z}}
\newcommand{\ud}[0]{\,\mathrm{d}}
\newcommand{\Exp}[0]{\mathbb{E}}
\newcommand{\ave}[1]{\langle #1\rangle}
\newcommand{\abs}[1]{|#1|}

\def\dsum{\displaystyle\sum}

\def\dint{\displaystyle\int}

\def\dfrac{\displaystyle\frac}

\def\r{\right}
\def\lf{\left}

\newtheorem{thm}{Theorem}[section]
\newtheorem{lem}{Lemma}[section]
\newtheorem{prop}{Proposition}[section]
\newtheorem{rem}{Remark}[section]
\newtheorem{cor}{Corollary}[section]
\newtheorem{defn}{Definition}[section]

\numberwithin{equation}{section}

\begin{document}

\arraycolsep=1pt

\title{{\vspace{-5cm}\small\hfill\bf Canad. J. Math., to appear}\\
\vspace{4.5cm}\Large\bf  Boundedness of Calder\'on-Zygmund
Operators on Non-homogeneous Metric Measure Spaces
\footnotetext{\hspace{-0.35cm} 2010 {\it Mathematics Subject
Classification}. {Primary 42B20; Secondary 42B25, 30L99.}
\endgraf{\it Key words and phrases.} Upper doubling,
geometrical doubling, dominating function, weak type $(1,1)$
estimate, Calder\'on-Zygmund operator, maximal operator.
\endgraf
The first author is supported by the Academy of Finland
(Grant Nos. 130166, 133264, 218148).
The third author is supported by National
Natural Science Foundation (Grant No. 10871025) of China
and Program for Changjiang Scholars and Innovative
Research Team in University of China.}}
\author{Tuomas Hyt\"onen, Suile Liu,
Dachun Yang\footnote{Corresponding author}\ \ and Dongyong Yang}
\date{ }
\maketitle

\begin{center}
\begin{minipage}{13.5cm}\small
{\noindent{\bf Abstract.}  {\small Let $({\mathcal X}, d, \mu)$ be a
separable metric measure space satisfying the known upper
doubling condition, the geometrical doubling condition and the
non-atomic condition that $\mu(\{x\})=0$ for all $x\in{\mathcal X}$.
In this paper, we show that the boundedness of a Calder\'on-Zygmund
operator $T$ on $L^2(\mu)$ is equivalent to that of $T$ on
$L^p(\mu)$ for some $p\in (1, \infty)$, and that of $T$ from $L^1(\mu)$
to $L^{1,\,\infty}(\mu).$ As an application, we prove that if $T$ is a
Calder\'on-Zygmund operator bounded on $L^2(\mu)$,
then its maximal operator is bounded on $L^p(\mu)$
for all $p\in (1, \infty)$ and from
the space of all complex-valued Borel measures on
${\mathcal X}$ to $L^{1,\,\infty}(\mu)$.
All these results generalize the corresponding results of Nazarov et al.
on metric spaces with
measures satisfying the so-called polynomial growth condition.}}
\end{minipage}
\end{center}

\section{Introduction}\label{s1}

\hskip\parindent
The classical theory of singular integrals of Calder\'on-Zygmund type started
with the study of convolution operators on the Euclidean space
associated with singular kernels and has been well developed into a
large branch of analysis
on metric spaces. One of the most interesting cases is the
``space of homogeneous type" in the sense of Coifman and Weiss \cite{cw71, cw77}.
Recall that a metric space $(\cx, d)$ equipped with a nonnegative Borel measure
$\mu$ is called a {\it space of homogeneous type} if $(\cx, d,
\mu)$ satisfies the following {\it measure doubling condition} that
there exists a positive constant $C_\mu$ such that for any ball
$B(x,r)\equiv \{y\in\cx:\,\, d(x, y)< r\}$ with $x\in\cx$ and $r\in(0, \fz)$,
\begin{equation}\label{1.1}
\mu(B(x, 2r))\le C_\mu \mu(B(x,r)).
\end{equation}
The measure doubling condition \eqref{1.1} was considered the cornerstone of
any extension to abstract frameworks of the theory of singular integrals.
However, recently, many results on the classical
Calder\'on-Zygmund theory have been proved still valid if the
measure doubling condition is replaced by a less demanding condition; see,
for example,  \cite{ntv2,t01, t01b, ntv03,t03,mm, b1} and the references therein.

In particular, let $\kz\in(0, \fz)$, $\cx$ be a separable metric space
endowed with a metric $d$ and a nonnegative
``$\kz$ dimensional" Borel measure $\mu$ in the sense that there exists a positive
constant $C_0$ such that for all $x\in\cx$ and $r\in(0, \fz)$,
\begin{equation}\label{1.2}
\mu(B(x,r))\le C_0r^\kz.
\end{equation}
Such a measure need not satisfy the doubling condition \eqref{1.1}.
In \cite{ntv2}, Nazarov, Treil and Volberg showed that if $T$ is a
Calder\'on-Zygmund operator bounded on $\lt$, then $T$ is bounded on
$\lp$ for all $p\in (1, \fz)$ and from $\lo$ to $\wlo$, and the
corresponding maximal operator $T^\sharp$ is also bounded on $\lp$
for any $p\in(1, \fz)$ and from the {\it space $\mathscr{M}(\cx)$}
of all complex-valued Borel measures on $\cx$ to $\wlo$; moreover,
Nazarov et al. \cite{ntv2} also proved that if $T$ is a
Calder\'on-Zygmund operator bounded from $\lo$ to $\wlo$, then $T$
is also bounded on $\lt$.

Notice that measures satisfying the polynomial growth condition
\eqref{1.2} are only different, not more general than measures
satisfying \eqref{1.1}. Thus, the Calder\'on-Zygmund theory with
non-doubling measures is not in all respects a generalization of the
corresponding theory of spaces of homogeneous type. In \cite{h10},
Hyt\"onen introduced a new class of metric measure spaces satisfying
the so-called upper doubling condition and the geometrical doubling
condition (see also Definitions \ref{d1.1} and  \ref{d1.2} below),
and a notion of the space of regularized $\mathop\mathrm{BMO}$.
This new class of metric measure spaces is a simultaneous
generalization of the spaces
of homogeneous type and metric spaces with power bounded measures.
Later, Hyt\"onen and Martikainen \cite{hm} further established a
version of $T(b)$ theorem for Calder\'on-Zygmund operators in such spaces.

Let $(\cx, d, \mu)$ be a separable metric space which satisfies the
upper doubling condition, the geometrical doubling condition and the
{\it non-atomic condition} that $\mu(\{x\})=0$ for all $x\in\cx$.
The goal of this paper is to  generalize the corresponding results
of Nazarov et al. in \cite{ntv2}. Precisely, in this paper, we show that
the boundedness of a Calder\'on-Zygmund operator $T$ on $L^2(\mu)$ is equivalent to
that of $T$ on $L^p(\mu)$ for some $p\in (1, \infty)$, and
that of $T$ from $L^1(\mu)$ to $L^{1,\,\infty}(\mu).$
As an application, we prove that if $T$ is a
Calder\'on-Zygmund operator bounded on $L^2(\mu)$,
then its maximal operator is bounded on $L^p(\mu)$
for all $p\in (1, \infty)$ and from
the space of all complex-valued Borel measures on
${\mathcal X}$ to $L^{1,\,\infty}(\mu)$.

To state our main results, we first
recall some necessary notions and notation.
We begin with the definition of the upper doubling spaces in \cite{h10}.

\begin{defn}\label{d1.1}\rm
A metric measure space $(\cx, d, \mu)$ is said to be {\it upper
doubling} if $\mu$ is a Borel measure on $\cx$ and there exists a
{\it dominating function} $\lz:\,\, \cx\times(0, \fz)\to (0, \fz)$
and a positive constant $C_\lz$ such that for each $x\in\cx$,
$r\to\lz(x, r)$ is non-decreasing, and for all $x\in\cx$ and
$r\in(0, \fz)$,
\begin{equation}\label{1.3}
\mu(B(x, r))\le\lz(x, r)\le C_\lz\lz(x, r/2).
\end{equation}
\end{defn}

\begin{rem}\label{r1.1}\rm
(i) Obviously, a space of homogeneous type is a special case of upper
doubling spaces, where one can take the dominating function $\lz(x,
r)\equiv \mu(B(x,r))$. On the other hand,  a metric space $(\cx, d, \mu)$
satisfying the  polynomial growth condition \eqref{1.2} (in particular,
$(\cx, d, \mu)\equiv (\rn, |\cdot|, \mu)$ with $\mu$ satisfying \eqref{1.2} for some
$\kz\in(0, n]$)  is also
an upper doubling measure space if we take $\lz(x, r)\equiv C_0r^\kz$.

(ii) Let $(\cx, d, \mu)$ be an upper doubling space and $\lz$ a
dominating function on $\cx\times(0, \fz)$ as in Definition
\ref{d1.1}. It was showed in \cite{hyy} that there exists
another dominating function $\wz\lz$
such that for all $x$, $y\in\cx$ with $d(x, y)\le r$,
\begin{equation}\label{1.4}
\wz\lz(x, r)\le\wz C\wz\lz(y, r).
\end{equation}
 Thus, in this paper, we {\it always assume} that $\lz$ satisfies \eqref{1.4}.
\end{rem}

We now recall the notion of geometrically doubling spaces introduced in \cite{h10}.

\begin{defn}\label{d1.2}\rm
A metric space $(\cx, d)$ is called {\it geometrically doubling}
if there exists some $N_0\in\nn\equiv\{1, 2, \cdots\}$ such that for
any ball $B(x, r)\subseteq\cx$, there exists a finite ball covering
$\{B(x_i, r/2)\}_i$ of $B(x, r)$ such that the cardinality of this
covering is at most $N_0$.
\end{defn}

\begin{rem}\label{r1.2}\rm
Let $(\cx, d)$ be a metric space. In \cite[Lemma 2.3]{h10},
Hyt\"onen showed that the following statements are mutually
equivalent:
\begin{itemize}
  \item [(i)] $(\cx, d)$ is geometrically doubling.
  \item [(ii)] For any $\ez\in (0, 1)$ and any ball
  $B(x, r)\subseteq\cx$, there exists a finite ball
  covering $\{B(x_i, \ez r)\}_i$ of $B(x, r)$ such that
  the cardinality of this covering is at most $N_0\ez^{-n}$, where
  and in what follows, $N_0$ is as in Definition \ref{d1.2} and
  $n\equiv\log_2N_0$.
  \item [(iii)] For every $\ez\in (0, 1)$, any ball
  $B(x, r)\subseteq\cx$ can contain at most
  $N_0\ez^{-n}$ centers $\{x_i\}_i$ of disjoint
  balls with radius $\ez r$.
  \item [(iv)]  There exists $M\in\nn$ such that any ball
  $B(x, r)\subseteq\cx$ can contain at most
  $M$ centers $\{x_i\}_i$ of disjoint
  balls $\{B(x_i, r/4)\}_{i=1}^M$.
\end{itemize}
\end{rem}

Now we recall the notions of standard kernels and corresponding
Calder\'on-Zygmund operators in the current setting from \cite{hm}.
Let $\mathscr{M}(\cx)$ be the {\it space of all complex-valued Borel
measures on $\cx$}. For a measure $\nu\in\mathscr{M}(\cx)$, we
denote by $\|\nu\|\equiv\int_\cx|d\nu(x)|$ the {\it total variation
of $\nu$} and $\supp\nu$ the {\it smallest closed set $F\subseteq\cx$
for which $\nu$
vanishes on $\cx\setminus F$} (such a smallest closed set always exists
since $\cx$ is separable; see \cite[p.\,466]{ntv2}).
Also, for any function $f$, $\supp f$ means the {\it essential support
of the function $f$, namely, the smallest closed set
$F\subseteq\cx$ such that $f$ vanishes at $\mu$-almost every $x\in \cx\setminus F$}.

\begin{defn}\rm\label{d1.3}
Let $\triangle\equiv\{(x, x):\,\, x\in\cx\}$. A {\it standard
kernel} is a mapping $K:\,\,\cx\times\cx\setminus \triangle\to \cc$
for which, there exist positive constants $\tau\in(0, 1]$ and
$C$ such that for all $x$, $y\in\cx$ with $x\not=y$,
\begin{equation}\label{1.5}
|K(x, y)|\le C\dfrac1{\lz(x, d(x, y))},
\end{equation}
and  that for all $x$, $\wz x$, $y\in\cx$ with $d(x, y)\ge 2d(x, \wz x)$,
\begin{equation}\label{1.6}
|K(x, y)-K(\wz x, y)|+|K(y, x)-K(y, \wz x)| \le C\dfrac{[d(x, \wz
x)]^\tau}{[d(x, y)]^\tau\lz(x, d(x, y))}.
\end{equation}

A linear operator $T$ is called a {\it Calder\'on-Zygmund operator}
with $K$ satisfying \eqref{1.5} and \eqref{1.6} if for all $f\in
L^\fz_b(\mu)$, the {\it space of bounded functions with bounded
support}, and $x\notin\supp f$,
\begin{equation*}
Tf(x)\equiv \dint_\cx K(x, y)f(y)\,d\mu(y).
\end{equation*}
\end{defn}

A new example of operators with kernel satisfying \eqref{1.5} and \eqref{1.6}
is the so-called Bergman-type operator appearing in \cite{vw09};
see also \cite{hm} for an explanation.

Assume that $T$ is a Calder\'on-Zygmund operator with $K$
satisfying \eqref{1.5} and \eqref{1.6}.
For any $\nu\in \mathscr{M}(\cx)$ with bounded support and
$x\in\cx\setminus \supp\nu$,
define
\begin{equation*}
T\nu(x)\equiv\dint_\cx K(x,y)\,d\nu(y).
\end{equation*}
Moreover, the maximal operator $T^\sharp$ associated with $T$ is
defined as follows. For every $f\in L^\fz_b(\mu)$ and $\nu\in
\mathscr{M}(\cx)$, we set, for all $x\in\mathcal {X}$,
\begin{equation*}
T^\sharp f(x)\equiv\sup\limits_{r>0}|T_rf(x)|
\end{equation*}
and
\begin{equation*}
T^\sharp \nu(x)\equiv\sup\limits_{r>0}|T_r\nu(x)|,
\end{equation*}
where for every $r>0$,
\begin{equation*}
T_rf(x)\equiv\dint_{d(x,\,y)> r}K(x,y)f(y)\,d\mu(y)
\end{equation*}
and
\begin{equation*}
T_r\nu(x)\equiv\dint_{d(x,\,y)> r}K(x,y)\,d\nu(y).
\end{equation*}

The main result of this paper reads as follows.

\begin{thm}\label{t1.1}
Let $T$ be a Calder\'on-Zygmund operator with kernel satisfying
\eqref{1.5} and \eqref{1.6}. Then the
following statements are equivalent:
\begin{itemize}
  \item[\rm(i)] $T$ is bounded on $\lt$; namely, there exists a positive
  constant $C$ such that for all $f\in\lt$,
  $$\|Tf\|_\lt\le C\|f\|_\lt.$$
  \item[\rm(ii)] $T$ is bounded on $\lp$ for some $p\in (1, \fz)$;
  namely, there exists a positive constant $C(p)$, depending on $p$, such that for all
  $f\in \lp$,
  $$\|Tf\|_\lp\le C(p)\|f\|_\lp.$$
  \item[\rm(iii)] $T$ is bounded from $\lo$ to $\wlo$; namely,
  there exists a positive constant $\wz C$ such that for all $f\in \lo$,
  \begin{equation}\label{1.7}
  \|Tf\|_\wlo\le \wz C\|f\|_\lo.
  \end{equation}
  \end{itemize}
\end{thm}

As an application of Theorem \ref{t1.1}, we also obtain the following
boundedness of the maximal operators associated with the Calder\'on-Zygmund
operators.

\begin{cor}\label{c1.1}
Let $T$ be a Calder\'on-Zygmund operator with kernel satisfying
\eqref{1.5} and \eqref{1.6}, which is bounded on $\lt$, and
$T^\sharp$ the maximal operator associated with $T$. Then the
following  statements hold:
\begin{itemize}
  \item [\rm(i)] Let $p\in (1, \fz)$. There exists a positive
  constant $c$ such that for all
  $f\in \lp$, $$\lf\|T^\sharp f\r\|_\lp\le c\|f\|_\lp.$$
  \item [\rm(ii)] There exists a positive constant $\wz c$ such that for all
  $\nu\in \mathscr{M}(\cx)$,
  \begin{equation}\label{1.8}
  \lf\|T^\sharp\nu\r\|_\wlo\le \wz c\|\nu\|.
  \end{equation}
  Moreover, for all $f\in\lo$,
\begin{equation}\label{1.9}
\lf\|T^{\sharp}f\r\|_{\wlo}\le \wz c\|f\|_\lo.
\end{equation}
  \end{itemize}
\end{cor}

Together, Theorem \ref{t1.1} and Corollary \ref{c1.1} consist
of a generalization of Nazarov--Treil--Volberg's
\cite[Theorems 1.1 and 10.1]{ntv2} from measures of type
\eqref{1.2} to general upper doubling measures.

This paper is organized as follows. Let $(\cx, d, \mu)$ be a
separable metric space satisfying Definitions \ref{d1.1} and
\ref{d1.2}, and the non-atomic condition. In Section \ref{s2}, we
make some preliminaries, including a Whitney-type Covering Lemma
\ref{l2.2} and a H\"ormander-type inequality, Lemma \ref{l2.4}. In
Section \ref{s3}, we first establish a Cotlar type inequality and an
endpoint estimate for $T$ in terms of the so-called \emph{elementary
measures}, which is an alternative to the Calder\'on-Zygmund
decomposition introduced by Nazarov, Treil and Volberg \cite{ntv2}
in the case of $\cx\equiv\rr^n$ and the polynomial bound
\eqref{1.2}. As an application of these estimates and the non-atomic
assumption, we further obtain (i) $\Rightarrow$ (ii),
(i) $\Rightarrow$ (iii) and (ii) $\Rightarrow$ (iii)
of Theorem \ref{t1.1}. We remark that the
non-atomic assumption is to guarantee that every $A\subseteq\cx$ of
positive $\mu$-measure can be further divided into two subsets, both
of positive $\mu$-measure (see Definition~\ref{d3.1} and
Remark~\ref{r3.2}). Notice that the non-atomic condition is
automatically true under the polynomial growth condition
\eqref{1.2}.

Section \ref{s4} is devoted to the proof of (iii) $\Rightarrow$ (i)
of Theorem \ref{t1.1}, while the proof of Corollary \ref{c1.1}
is presented in Section \ref{s5}. We point out that in \cite{ntv2},
the size condition of a given
Calder\'on-Zygmund kernel $K(x, y)$ is just related to the distance
$d(x, y)$ of $x$ and $y$, which is a very important fact used in
\cite{ntv2}. However, this may be false in our context, since $K(x,
y)$ is controlled by $[\lz(x, d(x, y))]^{-1}$ and $\lz(x, d(x, y))$
depends not only on $d(x, y)$, but also on $x$. To overcome this
difficulty, we first restrict $\mu$ to the closure of some ball,
$\overline B(x_0, M)$ for some fixed $x_0\in\cx$  and large radius
$M$, where and in what follows, for an open ball $B$, $\overline B$
means the {\it closure of} $B$, and show that (iii) $\Rightarrow$ (i)
of Theorem \ref{t1.1} holds for the restriction of $\mu$
with constant independent of $M$.
Then by a limiting argument we obtain (iii) $\Rightarrow$ (i)
of Theorem \ref{t1.1} for $\mu$.
Similar method is also used in the proof of Corollary \ref{c1.1}
in Section \ref{s5}. In
Section \ref{s5}, we also obtain an endpoint estimate for $T^\sharp$
via the elementary measures. Then as in \cite{ntv2}, using this and
some tools of probability theory, we establish Corollary \ref{c1.1}.

While this manuscript was in finishing touch, we learned that
(i) $\Rightarrow$ (ii) and (i) $\Rightarrow$ (iii)
of Theorem \ref{t1.1} and a variant of Lemma \ref{l3.1} in this paper
were also independently obtained by Anh and Duong  in \cite{bd} via
a different approach modeled after the work of Tolsa \cite{t01} for
measures of type \eqref{1.2} on $\rr^n$. In fact, Anh and Duong in
\cite{bd} first established a variant of the Calder\'on-Zygmund
decomposition in this setting; then as an application of this, Anh
and Duong further proved Theorem \ref{t1.1} and a variant of Lemma
\ref{l3.1}. Our approach, on the other hand, consists of extending
the techniques of Nazarov, Treil and Volberg \cite{ntv2}.

Finally, we make some conventions on symbols.
Throughout the paper, $C$, $\wz C$, $c$ and $\wz c$
stand for {\it positive constants} which
are independent of the main parameters, but they may vary from line to
line. Constants with subscripts,
such as $C_1$ and $c_1$, do not change in different
occurrences. Also, $C(\az,\bz,\cdots)$
denotes a positive constant depending on $\az,\bz,\cdots$.
If $f\le Cg$, we then write $f\ls g$ or $g\gs f$;
and if $f \ls g\ls f$, we then write $f\sim g.$
For any $q\in(1, \fz)$, let $q'\equiv q/(q-1)$ be
the {\it conjugate index} of $q$.
Sometimes, the {\it characteristic function} of
a set $E$ in $\cx$ is denoted by $\chi_E$ or $1_E$, depending on
what seems convenient in a particular place.
For $\rho\in(0, \fz)$ and $B\equiv B(x, r)$, the notation $\rho B\equiv B(x, \rho r)$
means the concentric dilation of $B$.
For any $f\in L^1_\loc(\mu)$, {\it its average in a set $E$}  is denoted by
$$\ave{f}_E\equiv \frac1{\mu(E)}\dint_Ef(x)\,d\mu(x).$$

\section{Preliminaries}\label{s2}

\hskip\parindent
In this section, we make some preliminary lemmas used in the rest of the paper.
We begin with a covering lemma in \cite{hyy}, which is a simple corollary
of \cite[Theorem 1.2]{he} and \cite[Lemma 2.5]{h10}.

\begin{lem}\label{l2.1}
Let $(\cx, d)$ be a geometrically doubling metric space. Then every
family ${\mathcal F}$ of balls of uniformly bounded diameter contains
an at most countable disjointed subfamily ${\mathcal G}$ such that
$\cup_{B\in {\mathcal F}}B\subseteq\cup_{B\in {\mathcal G}}5B$.
\end{lem}

The following  Whitney type covering lemma was included in
\cite[p.\,70,\,Theorem (1.3)]{cw71} (see also \cite[p.\,623,\,Theorem
(3.2)]{cw77} or \cite{b1}), we present the proof here for completeness.

\begin{lem}\label{l2.2}
Let $\Omega\subsetneq \cx$ be a bounded open set.
Then there exists a sequence $\{B_i\}_i$ of balls such that
\begin{itemize}
  \item [${\rm (w)_i}$] $\Omega=\cup_i B_i$ and $2B_i\subseteq\Omega$ for all $i$;
  \item [${\rm (w)_{ii}}$] there exists a positive constant $C$
  such that for all $x\in\cx$, $\sum_i \chi_{B_i}(x)\le C$;
  \item [${\rm (w)_{iii}}$] for all $i$, $(3B_i)
  \cap(\cx\setminus \Omega)\not=\emptyset$.
\end{itemize}
\end{lem}

\begin{proof} For any $x\in \Omega$, let
$\hat r(x)\equiv\frac{1}{10}\dist(x,\cx\setminus \Omega)$, where and
in what follows, for any $y$ and set $E$, $\dist(y,
E)\equiv\inf_{z\in E}d(y, z)$. The function $\hat r(x)$ is strictly
positive because $\Omega$ is open and the balls centered at $x$
form a basis of neighborhood of $x$. Then by Lemma \ref{l2.1},
there exists a sequence $\{\hat B_i\}_i\equiv\{B(x_i, \hat r(x_i))\}_i$
of balls with $\{x_i\}_i\subseteq\boz$ satisfying that $\{\hat
B_i\}_i$ are pairwise disjoint and $\{B_i\}_i\equiv\{5\hat B_i\}_i$
forms a covering of $\Omega$. Moreover, for each $i$, set $r_i\equiv
5\hat r(x_i)$. Then for any $i$ and $y\in 2B_i$, since
$\cx\setminus\Omega$ is closed, we have that
\begin{equation*}
\dist(y, \cx\setminus\Omega)\ge \dist(x_i, \cx\setminus \Omega)-
d(y, x_i)>\dist(x_i, \cx\setminus \Omega)-2r_i=0.
\end{equation*}
This yields $y\in \Omega$ and hence $2B_i\subseteq\Omega$, which
implies $\rm(w)_i$. On the other hand, since, by the definition of $r_i$,
$3r_i=\frac32 \dist(x_i,\cx\setminus\Omega)$, we then see that
$(3B_i)\cap(\cx\setminus \Omega)\not=\emptyset$, which implies
${\rm (w)_{iii}}$.

It remains to show ${\rm (w)_{ii}}$. To this end, we claim
that for any $i$ and $x\in B_i\cap \Omega$,
\begin{equation}\label{2.1}
\frac{1}3\dist(x,\cx\setminus \Omega)< r_i< \dist(x,\cx\setminus
\Omega).
\end{equation}
Indeed, by the fact that $\cx\setminus\Omega$
is closed, we have
$$\dist(x_i,\cx\setminus
\Omega)\le \dist(x,\cx\setminus \Omega)+d(x,x_i),$$
which further implies that
\begin{equation}\label{2.2}
\dist(x_i,\cx\setminus\Omega)-r_i< \dist(x,\cx\setminus\Omega).
\end{equation}
Observe that by the definition of $r_i$, $\dist(x_i,\cx\setminus\Omega)=2r_i.$
This together with \eqref{2.2} gives us that
\begin{equation}\label{2.3}
r_i< \dist(x,\cx\setminus\Omega).
\end{equation}
On the other hand, by this, we also have
$$\dist(x,\cx\setminus\Omega)\le d(x,x_i)+
\dist(x_i,\cx\setminus\Omega)< 3r_i,$$
which combined with \eqref{2.3} implies \eqref{2.1}, and hence the claim holds.

Now let $x\in\Omega$ and $B_i$ contain $x$. Then by \eqref{2.1}, we
see that $B_i\subseteq B(x, 2\dist(x, \cx\setminus\Omega))$. On the
other hand, observe that $\{\frac15B_i\}_i= \{\hat B_i\}_i$ are
mutually disjoint. This together with another application of
\eqref{2.1} implies that $\{B(x_i, \frac1{15}\dist(x,
\cx\setminus\Omega))\}_i$ are also pairwise disjoint. From this and
Remark \ref{r1.2}(iii), we deduce that the cardinality of
$$\lf\{B\lf(x_i,\frac1{15}\dist(x, \cx\setminus\Omega)\r)\r\}_i$$
contained in $B(x, 2\dist(x, \cx\setminus\Omega))$ is
at most $N_0{30}^n$, and so is the
cardinality of $\{B_i\}_i$ containing $x$. Thus, ${\rm (w)_{ii}}$ holds,
which completes the proof of Lemma \ref{l2.2}.
\end{proof}

Let $p\in(0, \fz)$, $f\in L^p_\loc(\mu)$ and
$\nu\in\mathscr{M}(\cx)$. The {\it centered maximal functions}
$\cm_p f$ and $\cm \nu$ are defined  by setting, for all $x\in\cx$,
\begin{equation*}
\mathcal{M}_pf(x)\equiv \sup_{r>0}
\left[\dfrac{1}{\mu(\overline B(x,5\,r))}\int_{\overline
B(x,\,r)}|f(y)|^{p} \,d\mu(y)\right]^{\frac{1}{p}}
\end{equation*}
and
$$\cm\nu(x)\equiv\sup_{r>0}\frac{\nu(\overline B(x, r))}{\mu(\overline B(x, 5r))}.$$
 If $p=1$, we denote $\cm_1$ simply by $\cm$, which is called
the {\it centered Hardy-Littlewood maximal operator}.

\begin{lem}\label{l2.3}  The following statements hold:
\begin{itemize}
  \item [{\rm (i)}] Let $p\in[1, \fz)$. Then $\cm_p$ is bounded on $\lq$ for all
$q\in(p, \fz]$ and from $\lp$ to $\wlp$.
  \item [{\rm (ii)}] Let $p\in (0, 1)$. Then $\cm_p$ is bounded on $\wlo$.
  \item [{\rm (iii)}] There exists a positive constant $C$ such that for
  all $\nu\in \mathscr{M}(\cx)$, $\cm \nu\in\wlo$ and
$$  \|\cm \nu\|_\wlo\le C\|\nu\|.$$
\end{itemize}
\end{lem}

\begin{proof}
The proof of (ii) is just a mimic of the one in \cite[Lemma
3.2]{ntv2}, and the proof of (iii) is similar to that of boundedness
of $\cm$ from $\lo$ to $\wlo$ in (i). Thus, it suffices to prove (i)
by similarity. By Lemma 2.5 in \cite{h10}, any disjoint collection
of open balls is at most countable, so is any disjoint collection of
closed balls. Moreover, by an argument similar to that used in the
proof of Proposition 3.5 in \cite{h10}, we see that $\cm_p$ is
bounded on $\lq$ for all
 $q\in(p, \fz]$ and bounded from $\lp$ to $\wlp$. This finishes the proof of
 Lemma \ref{l2.3}.
\end{proof}

\begin{lem}\label{l2.4}
Let $\eta\in\mathscr{M}(\cx)$ such that
$\eta(\cx)=0\,and\,\supp\,\eta\subseteq \overline B(x,\rho)$ for
some $\rho\in(0, \infty)$ and $x\in\cx$, and $T$ be a
Calder\'on-Zygmund operator with kernel satisfying
\eqref{1.5} and \eqref{1.6} as in Definition \ref{d1.3}. Then there
exists a positive constant $C$, independent of $\eta$, $x$ and
$\rho$, such that for all nonnegative Borel measures $\nu$ on $\cx$,
\begin{equation}\label{2.4}
\int_{\cx\setminus B(x,\,2\rho)}|T\eta(y)|
\,d\nu(y)\le C\|\eta\|\cm\nu(x).
\end{equation}
Moreover, for any $p\in[1, \fz)$ and $f\in L^p_\loc(\mu)$,
\begin{equation}\label{2.5}
\int_{\cx\setminus B(x,\,2\rho)}|T\eta(y)||f(y)|
\,d\mu(y)\le C\|\eta\|\cm_pf(x)
\end{equation}
and
\begin{equation}\label{2.6}
\int_{\cx\setminus B(x,\,2\rho)}|T\eta(y)|\,d\mu(y)\le C\|\eta\|,
\end{equation}
where $C$ is a positive constant, independent of $\eta$, $x$, $\rho$
and $f$.
\end{lem}

\begin{proof}
By similarity, we only prove \eqref{2.4}. By $\eta(\cx)=0$,
$\supp\,\eta\subseteq \overline B(x,\rho)$ and \eqref{1.6}, we have
that for any $y\in\cx\setminus B(x,\,2\rho)$,
\begin{eqnarray*}
|T\eta(y)| &&=\left|\int_{\overline B(x,\,\rho)}K(y,\wz{x})\,d\eta(\wz{x})\right|
=\left|\int_{\overline B(x,\,\rho)}[K(y,\wz{x})-K(y,x)]\,d\eta(\wz{x})\right|\\
&&\le\|\eta\|\sup_{\wz{x}\in \overline B(x,\,\rho)}
|K(y,\wz{x})-K(y,x)|\ls\|\eta\|\left[\frac{\rho}{d(x, y)}\right]^{\tau}
\frac{1}{\lz(x,d(x,y))}.
\end{eqnarray*}
Therefore, by \eqref{1.3}, we have that
\begin{eqnarray*}
\dint_{\cx\setminus B(x,\, 2\rho)}|T\eta(y)|\,d\nu(y)
&&\ls\|\eta\|\int_{\cx\setminus B(x,\,2\rho)}\lf[\frac{\rho}{d(x,
y)}\r]^{\tau}
\frac{1}{\lz(x,d(x,y))}\,d\nu(y)\\
&&\ls\|\eta\|\sum_{k=1}^\fz\int_{B(x,\, 2^{k+1}\,\rho)\setminus
B(x,\,2^{k}\rho)}\frac{1}{2^{k\tau}}
\frac{1}{\lz(x, 2^{k}\rho)}\,d\nu(y)\\
&&\ls\|\eta\|\sum_{k=1}^\fz\frac{1}{2^{k\tau}}
\frac{\nu(B(x, 2^{k+1}\rho))}{\mu(B(x, 5\cdot 2^{k+1}\rho))}\\
&&\ls\|\eta\|\sum_{k=1}^\fz\frac{1}{2^{k\tau}}
\cm\nu(x)\ls\|\eta\|\cm\nu(x),
\end{eqnarray*}
which completes the proof of Lemma \ref{l2.4}.
\end{proof}

\section{Proof of Theorem \ref{t1.1}, Part I}\label{s3}

\hskip\parindent
This section is devoted to the proof of the implicity
${\rm (i)}\Rightarrow{\rm (ii)}$,
${\rm (i)}\Rightarrow{\rm (iii)}$ and
${\rm (ii)}\Rightarrow{\rm (iii)}$ of Theorem \ref{t1.1}.
To this end,  we first establish an endpoint estimate for $T$ via the so-called
elementary measures which are finite linear combinations of unit point
masses with positive coefficients.
We begin with the following Cotlar type inequality inspired by \cite{ntv2}.

\begin{lem}\label{l3.1}
Let $T$ be a Calder\'on-Zygmund operator with kernel satisfying
\eqref{1.5} and \eqref{1.6}, which is bounded on $\lt$.
Then there exist positive
constants $C$ and $c$ such that for any $f\in L^\fz_b(\mu)$ and $x\in\supp\mu$,
\begin{equation}\label{3.1}
T^\sharp(f)(x)\le C\cm(Tf)(x)+c \cm_2(f)(x).
\end{equation}
\end{lem}

\begin{proof}
Let $x\in\supp\mu$, $r\in(0,\fz)$, and $r_j\equiv5^j\,r$ and
$\mu_j\equiv\mu(\overline B(x,r_j))$ for
$j\in\zz_+\equiv\nn\cup\{0\}$. We claim that there exists some
$j\in\nn$ such that $\mu_{j+1}\le 4C_\lz^6\mu_{j-1}$, where $C_\lz$
is as in \eqref{1.3}. For otherwise, by \eqref{1.3}, we would have
that for every $j\in\nn$,
\begin{eqnarray*}
\mu_0 <\lf(4C_\lz^6\r)^{-j}\mu_{2j}=\lf(4C_\lz^6\r)^{-j}\mu
\lf(\overline B\lf(x,r_{2j}\r)\r)
\ls\lf(4C_\lz^6\r)^{-j}\lz\lf(x,5^{2j}r\r)\ls5^{-j}\lz(x,r).
\end{eqnarray*}
Letting $j\to0$, we have $\mu (\overline B(x,r))=0$, which contradicts to the fact
that $\mu (\overline B(x,r))>0$ for each $r>0$ and each $x\in\supp\mu$.
 Thus, the claim holds.

Let $k\in\nn$ be the {\it smallest integer} such that
$\mu_{k+1}\le4C_\lz^6\mu_{k-1}$ and $R\equiv r_{k-1}\equiv5^{k-1}r$.
Then we see that
\begin{equation}\label{3.2}
\mu\lf(\overline B\lf(x, 25R\r)\r)\ls \mu\lf(\overline B\lf(x, R\r)\r).
\end{equation}
Observe that for all $j\in\{1,\cdots,k\}$, we have that
$\mu_{j+1}\le(2C_\lz^3)^{j+2-k}\mu_{k}$ and
$$\lz(x,r_k)\le(C_\lz^3)^{\max\{0, k-j-1\}}\lz(x, r_{j+1}).$$
Let $f\in L_b^\infty(\mu)$. From
this, \eqref{1.5}, \eqref{1.3} and the H\"older inequality, we then
deduce that
\begin{eqnarray}\label{3.3}
|T_rf(x)-T_{5R}f(x)|&&\le\dint_{\overline B(x,\,5R)\setminus
\overline B(x,\,r)}|K(x,y)||f(y)|\,d\mu(y)\\
&&=\dsum_{j=1}^k\int_{\overline B(x,\,r_j)\setminus
\overline B(x,\,r_{j-1})}|K(x,y)||f(y)|\,d\mu(y)\nonumber\\
&&\ls\dsum_{j=1}^k\dfrac{\mu(\overline B(x,r_{j+1}))}
{\lz(x, r_{j+1})}\cm(f)(x)\nonumber\\
&&\ls\dsum_{j=1}^k2^{j-k}\cm(f)(x)\ls\cm(f)(x).\nonumber
\end{eqnarray}

Let
$$V_R(x)\equiv\frac{1}{\mu(\overline B(x,R))}\int_{\overline B(x,R)}Tf(y)\,d\mu(y).$$
Then we have
\begin{equation}\label{3.4}
|V_R(x)|\ls\cm(Tf)(x).
\end{equation}
On the other hand, observe that
\begin{eqnarray*}
T_{5R}f(x)&&= \int_{\cx\setminus
\overline B(x,\,5R)}K(x,y)f(y)\,d\mu(y)
=\int_{\cx}K(x, y)\chi_{\cx\setminus \overline B(x,\,5R)}(y)f(y)\,d\mu(y)\nonumber\\
&&=T\lf(f\chi_{\cx\setminus \overline B(x,\,5R)}\r)(x)
=\lf\langle\dz_x, T\lf(f\chi_{\cx\setminus \overline B(x,\,5R)}\r)\r\rangle\nonumber\\
&&=\lf\langle T^{*}\dz_x, f\chi_{\cx\setminus
\overline B(x,\,5R)}\r\rangle=\int_{\cx\setminus
\overline B(x,\,5R)}T^{*}\dz_x(y)f(y)\,d\mu(y)\nonumber,
\end{eqnarray*}
where and in what follows, $\dz_x$ denotes the {\it Dirac measure at} $x$, and
for a linear operator $T$, $T^\ast$ means the {\it adjoint operator of} $T$.
By writing
\begin{eqnarray*}
V_R(x)&&=\frac{1}{\mu(\overline B(x,R))}\int_{\cx}
\chi_{\overline B(x,\,R)}(y)T(f)(y)\,d\mu(y)\\
&&=\frac{1}{\mu(\overline B(x,R))}\int_{\cx}
\chi_{\overline B(x,\,R)}(y)T\lf(f\chi_{\overline B(x,\,5R)}\r)(y)\,d\mu(y)\nonumber\\
&&\hs+\int_\cx T^{*}\lf(\frac{\chi_{\overline B(x,\, R)}}{\mu(\overline B(x,
R))}\r)(y)f(y)\chi_{\cx\setminus \overline B(x,\,5R)}(y)\,d\mu(y),\nonumber
\end{eqnarray*}
we obtain that
\begin{eqnarray}\label{3.5}
|T_{5R}f(x)-V_R(x)|&& \le\lf|\int_{\cx\setminus \overline
B(x,\,5R)}T^\ast\lf(\dz_x-\frac{\chi_{\overline B(x, R)}}
{\mu(\overline B(x,R))}\,d\mu\r)(y)f(y)\,d\mu(y)\r|\\
&&\hs+\lf|\frac{1}{\mu(\overline B(x,R))}\int_\cx\lf[Tf\chi_{
\overline B(x,\,5R)}(y)\r]\chi_{\overline B(x,\,R)}(y)\,d\mu(y)\r|\nonumber\\
&&\equiv {\rm L}_1+{\rm L}_2.\nonumber
\end{eqnarray}
By \eqref{2.5}, we have ${\rm L}_1\ls\cm(f)(x).$
From the H\"older inequality, the boundedness of $T$  on $\lt$ and \eqref{3.2}.
we further deduce that
\begin{eqnarray*}
{\rm L}_2&&\le
\lf[\mu\lf(\overline B\lf(x,R\r)\r)\r]^{-\frac{1}{2}}
\lf[\int_\cx\lf|T\lf(f\chi_{\overline B(x,\,5R)}\r)(y)
\r|^2\,d\mu(y)\r]^{\frac{1}{2}}\\
&&\ls\lf[\mu\lf(\overline B(x,R)\r)\r]^{-\frac{1}{2}}
\lf[\int_{\overline B(x,\,5R)}|f(y)|^2
\,d\mu(y)\r]^{\frac{1}{2}}\ls\cm_2(f)(x).\nonumber
\end{eqnarray*}
Then combining the estimates for ${\rm L}_1$ and ${\rm L}_2$,
and using \eqref{3.5}, \eqref{3.4} and \eqref{3.3},
 we have that for any $r\in(0,\fz)$,
\begin{eqnarray*}
\lf|T_rf(x)\r|&&\le\lf|T_rf(x)-T_{5R}f(x)\r|+\lf|T_{5R}f(x)-V_R(x)\r|+\lf|V_R(x)\r|\\
&&\ls \cm_2(f)(x)+\cm(Tf)(x).
\end{eqnarray*}
Taking the supremum over $r\in(0,\fz)$, we obtain \eqref{3.1}, and
hence complete the proof of Lemma \ref{l3.1}.
\end{proof}

\begin{rem}\label{r3.1}
We point out that if we replace the boundedness of $T$ on $\lt$
in Lemma \ref{l3.1} by the boundedness of $T$ on
$\lq$ for some $q\in (1, \fz)$, then \eqref{3.1} still holds with
$\cm_2$ replaced by $\cm_q$.
\end{rem}

To prove Theorem \ref{t1.1}, we still need to recall the notion
of non-atomic space; see,
for example, \cite{g08}.

\begin{defn}\label{d3.1}\rm
A subset $A$ of a measure space $(\cx,\mu)$ is called an {\it atom} if
$\mu(A)>0$ and each $B\subseteq\,A$ has measure either equal to zero
or equal to $\mu(A)$. A measure space $(\cx,\mu)$ is called {\it
non-atomic} if it contains no atoms.
\end{defn}

\begin{rem}\label{r3.2}\rm
We know from Definition \ref{d3.1} that $\cx$ is non-atomic if and
only if for any $A\subseteq\cx$ with $\mu(A)>0$, there exists a proper
subset $B\subsetneq A$ with $\mu(B)>0$ and $\mu(A\setminus B)>0.$ By
this, it is straightforward that if $\mu(\{x\})=0$ for any $x\in\cx$, then
$(\cx,\mu)$ is a non-atomic space. Moreover, it is known that if
$(\cx, \mu)$ is a non-atomic measure space, then for any sets
$A_0\subseteq A_1\subseteq \cx$ such that $0<\mu(A_1)<\fz$ and
$\mu(A_0)\le t\le \mu(A_1)$ for some $t\in(0, \fz)$, there exists a
set $E$ such that $A_0\subseteq E\subseteq A_1$ and $\mu(E)=t$; see,
for example, \cite[p.\,65]{g08}.
\end{rem}

We say that $\nu$ is an \emph{elementary measure} if it is of the form
\begin{equation*}
   \nu\equiv\sum_{i=1}^N\az_i\dz_{x_i},
\end{equation*}
where $N\in\nn$, $\dz_{x_i}$ is the Dirac measure at some
$x_i\in\cx$ and $\az_i>0$ for $i\in\{1,\cdots,N\}$.
To prove Theorem \ref{t1.1}, we first establish an endpoint estimate for
$T$ on these elementary measures. This generalizes
\cite[Theorem~5.1]{ntv2}, where it was proven for polynomially
bounded measures as in \eqref{1.2} on $\rr^n$.

\begin{thm}\label{t3.1}
Let $T$ be a Calder\'on-Zygmund operator with kernel satisfying
\eqref{1.5} and \eqref{1.6}, which is bounded on $\lt$. Then there exist positive
constants $C_1$ and $C_2$ such that for all elementary measures $\nu$,
\begin{equation}\label{3.6}
\|T\nu\|_{L^{1,\,\infty}(\mu)}\le \lf[C_1+C_2\|T\|_{\lt\to\lt}\r]\|\nu\|.
\end{equation}
\end{thm}

\begin{proof}
Without loss of generality, we may normalize $\nu$ such that
$\|\nu\|=\sum_{i=1}^N\az_i=1$, and hence we only need prove
\begin{equation}\label{3.7}
\|T\nu\|_{\wlo}\le C_1+C_2\|T\|_{\lt\to\lt}.
\end{equation}
Since for $t\in(0, 1/\mu(\cx)]$, we
have
$$t\mu(\{x\in\cx:\,|T\nu(x)|>t\})\le t\mu(\cx)\le1.$$
Therefore it remains to consider the case $t\in(1/\mu(\cx), \fz)$.
Let $\overline B(x_1,\rho_1)$  be the
{\it smallest closed ball} such that $\mu(\overline B(x_1, \rho_1))\ge\az_1/t$.
Indeed, since the function $\rho\to\mu(\overline B(x,\rho))$ is
increasing and continuous from the right, and greater than $1/t\ge\az_1/t$ for
sufficiently large $\rho>0$, such $\rho_1$ exists and is strictly
positive. Then
$$\mu(B(x_1,\rho_1))=\lim\limits_{\rho\to\rho_1-0}
\mu(\overline B(x_1,\rho))\le\frac{\az_1}{t}.$$
Since $(\cx,\mu)$ is non-atomic, by Remark \ref{r3.2},
we can find a Borel set $E_1$ such that
$$B(x_1,\rho_1)\subseteq E_1\subseteq
\overline B(x_1,\rho_1)$$
and $\mu(E_1)=\frac{\az_1}{t}$.

Let $\overline B(x_2,\rho_2)$ be the {\it smallest closed ball} such that
$\mu(\overline B(x_2,\rho_2)\setminus E_1)\ge\az_2/t.$ Similarly, for the
corresponding open ball $B(x_2,\rho_2)$, we have
$\mu(B(x_2,\rho_2)\setminus E_1)\le\az_2/t$ and henceforth find a
Borel set $E_2$ with the property:
$$\lf(B(x_2,\rho_2)\setminus E_1\r)\subseteq E_2\subseteq
\lf(\overline B(x_2,\rho_2)\setminus E_1\r)$$
and $\mu(E_2)=\frac{\az_2}{t}$.

Repeating the process, for $i\in\{3,\cdots,N\}$, we have $\overline
B(x_i, \rho_i)$ and $E_i$ such that $\overline B(x_i,\rho_i)$ is the
{\it smallest closed ball} satisfying that $\mu(\overline
B(x_i,\rho_i)\setminus\bigcup_{l=1}^{i-1}E_l)\ge\az_i/t$,
$$\lf(B(x_i,\rho_i)\setminus\bigcup
\limits_{l=1}^{i-1}E_l\r)\subseteq E_i\subseteq \lf(\overline
B(x_i,\rho_i)
\setminus\bigcup\limits_{l=1}^{i-1}E_l\r)$$
and $\mu(E_i)=\frac{\az_i}{t}$. Let $E\equiv\bigcup_{i=1}^N E_i$. Then
by the fact that $\sum_{i=1}^N\az_i=1$ together with the choices of
$\{B(x_i, \rho_i)\}_{i=1}^N$ and $\{E_i\}_{i=1}^N$, we see that
$$\bigcup\limits_{i=1}^N B(x_i,\rho_i)\subseteq E\subseteq
\bigcup\limits_{i=1}^N\overline
B(x_i,\rho_i)$$
and $\mu(E)=\frac{1}{t}$.

Outside $E$, let us compare $T\nu$ to $t\sigma$, where
\begin{equation*}
  \sigma\equiv \sum_{i=1}^N\chi_{\cx\setminus \overline
B(x_i,\,2\rho_i)} T(\chi_{E_i}\,d\mu).
\end{equation*}
We have
\begin{eqnarray}\label{3.8}
T\nu-t\sz&& =T\lf(\sum\limits_{i=1}^N\az_i\dz_{x_i}\r)
-t\sum\limits_{i=1}^N\chi_{\cx\setminus
\overline B(x_i,\,2\rho_i)}T(\chi_{E_i}\,d\mu)\\
&&=\sum\limits_{i=1}^N\lf[\az_i\,T\delta_{x_i}-
t\chi_{\cx\setminus\overline B(x_i,\,2\rho_i)}
T(\chi_{E_i}\,d\mu)\r]\equiv\sum\limits_{i=1}^N\fai_i.\nonumber
\end{eqnarray}
Notice that for any $i$,
\begin{eqnarray}\label{3.9}
&&\dint_{\cx\setminus E}|\fai_i(x)|\,d\mu(x)\\
&&\hs=\dint_{\cx\setminus\bigcup\limits_{i=1}^{N}E_i}\lf|\az_i\,T\dz
_{x_i}(x)- t\chi_{\cx\setminus\overline B(x_i,\,2\rho_i)}(x)
T(\chi_{E_i}\,d\mu)(x)\r|\,d\mu(x)\noz\\
&&\hs\le\dint_{\cx\setminus \overline
B(x_i,\,2\rho_i)}\lf|\az_i\,T\dz_{x_i}(x)-t
\chi_{\cx\setminus \overline B(x_i,\,2\rho_i)}(x)
T(\chi_{E_i}\,d\mu)(x)\r|\,d\mu(x)\noz\\
&&\hs\hs+\dint_{\overline B(x_i,\,2\rho_i) \setminus
B(x_i,\,\rho_i)}\cdots\noz\\
&&\hs=\dint_{\cx\setminus \overline B
(x_i,\,2\rho_i)}\lf|T(\az_i\dz_{x_i}-t\chi_{E_i}\,d\mu)(x)\r|\,d\mu(x)\noz\\
&&\hs\hs+\dint_{\overline B(x_i,\,2\rho_i) \setminus
B(x_i,\,\rho_i)}\az_i|T\dz_{x_i}(x)| \,d\mu(x)\equiv {\rm J}_1+{\rm J}_2.\noz
\end{eqnarray}
For each $i$, using \eqref{2.6} and
$\mu(E_i)=\frac{\az_i}t$, we see that
$${\rm J}_1\ls\|\az_i\dz_{x_i}-t\chi_{E_i}\,d\mu\|\ls\az_i.$$
Moreover, from \eqref{1.5}, \eqref{1.4} and \eqref{1.3}, we deduce
that
\begin{eqnarray*}
{\rm J}_2&&\ls\int_{\overline
B(x_i,\,2\rho_i)\setminus B(x_i,\,\rho_i)}
\dfrac{\az_i}{\lz(x, d(x, x_i))}\,d\mu(x)\\
&&\ls\int_{\overline B(x_i,\,2\rho_i)\setminus
B(x_i,\,\rho_i)}\dfrac{\az_i}{\lz(x_i, d(x, x_i))}\,d\mu(x)
\ls\az_i\dfrac{\mu(\overline B(x_i,\,2\rho_i))}{\lz(x_i,
\rho_i)}\ls\az_i.
\end{eqnarray*}
By the estimates of ${\rm J}_1$ together with ${\rm J}_2$ and \eqref{3.9}, we
obtain that $\int_{\cx\setminus E}|\varphi_i|\,d\mu\ls\az_i,$
which, together with \eqref{3.8} and the fact that $\sum_{i=1}^N\az_i=1,$
further implies that there exists a
positive constant $C_3$ such that
\begin{equation}\label{3.10}
\int_{\cx\setminus E}|T\nu(x)-t\sz(x)|\,d\mu(x)\le\sum_{i=1}^N
\int_{\cx\setminus E}|\fai_i(x)|\,d\mu(x)\le C_3.
\end{equation}

Via \eqref{3.10}, to accomplish the proof of Theorem \ref{t3.1}, it
suffices to show that there exist positive constants $C_4 $ and
$C_5$ such that  $C_6 \equiv C_4 +C_5\|T\|_{\lt\to\lt}$ satisfying
\begin{equation}\label{3.11}
\mu(\{x\in\cx:\,\, |\sigma(x)|>C_6 \})\le \frac2t.
\end{equation}
Indeed, assume that \eqref{3.11} holds for the moment. Then from
$\mu(E)=\frac1t$, \eqref{3.10} and \eqref{3.11}, we deduce that
\begin{eqnarray*}
&&\mu\lf(\lf\{x\in\cx:\,|T\nu(x)|>(C_3+C_6 )t\r\}\r)\\
&&\hs\le\mu\lf(\{x\in\cx\setminus E:\,|T\nu(x)|>(C_3+C_6 )t\}\r)+\mu(E)\\
&&\hs\le\mu\lf(\{x\in\cx\setminus E:\,|T\nu(x)-t\sz(x)|>C_3t\}\r)\\
&&\hs\hs+\mu\lf(\{x\in\cx:\,|\sz(x)|>C_6\}\r)+\mu(E)\le \frac4t.
\end{eqnarray*}
This implies \eqref{3.7}, and hence finishes the proof of Theorem
\ref{t3.1}, up to the verification of \eqref{3.11},
which we do in the following lemma.
\end{proof}

\begin{lem}
The estimate \eqref{3.11} holds.
\end{lem}

\begin{proof}
We first claim that
there exist $C_4 $ and $C_5$ such that  for any set $F$ with
$\mu(F)=\frac1t$,
\begin{equation}\label{3.12}
\lf|\int_\cx\sz(x)\chi_F(x)\,d\mu(x)\r|\le
\frac{1}{t}\lf[C_4 +C_5\|T\|_{\lt\to\lt}\r].
\end{equation}
Indeed, let $F$ be such a set. Then the definition of $\sigma$ gives us that
\begin{eqnarray}\label{3.13}
\dint_\cx\sz(x)\chi_F(x)\,d\mu(x)
&&=\sum_{i=1}^N\int_\cx T\chi_{E_i}(x)
\chi_{F\setminus\overline B(x_i,\,2\rho_i)}(x)\,d\mu(x)\\
&&=\sum_{i=1}^N\int_\cx\chi_{E_i}(x) T^{*}
\chi_{F\setminus\overline
B(x_i,\,2\rho_i)}(x)\,d\mu(x).\nonumber
\end{eqnarray}
From \eqref{1.4} and \eqref{1.3}, it follows that for all
$x\in E_i\subseteq\overline B(x_i, \rho_i)$ and $y\in\overline
B(x_i,\,2\rho_i)\setminus\overline B(x,\,\rho_i)$, $\lz(x_i, \rho_i)\ls
\lz(y, d(x, y))$, which, together with \eqref{1.5} and \eqref{1.4},
further implies that for all $x\in E_i\subseteq\overline B(x_i,
\rho_i)$,
 \begin{eqnarray*}
\lf|T^{*}\chi_{F\setminus\overline B(x_i,\,2\rho_i)}(x)-T^{*}
\chi_{F\setminus\overline B(x,\,\rho_i)}(x)\r|&&\le
\int_{\overline B(x_i,\,2\rho_i)\setminus
\overline B(x,\,\rho_i)}|K(y,x)|\,d\mu(y)\\
&&\ls\int_{\overline B(x_i,\,2\rho_i)\setminus
\overline B(x,\,\rho_i)}\dfrac1{\lz(y, d(x, y))}\,d\mu(y)\\
&&\ls\frac{\mu(\overline B(x_i, 2\rho_i))} {\lz(x_i,\rho_i)}\ls1.
\end{eqnarray*}
This combined with the fact that $T^\ast\chi_{F\setminus\overline B(x,
\,\rho_i)}(x) \le (T^\ast)^\sharp\chi_{F}(x)$ and Lemma \ref{l3.1}  yields
that for all $x\in
E_i\subseteq\overline B(x_i, \rho_i)$,
\begin{eqnarray*}
\lf|T^{*}\chi_{F\setminus\overline B(x_i,\,2\rho_i)}(x)\r|
&&\le\lf|T^{*}\chi_{F\setminus\overline B(x_i,\,2\rho_i)}(x)-T^{*}
\chi_{F\setminus\overline B(x,\,\rho_i)}(x)\r|
+\lf|T^{*}\chi_{F\setminus\overline B(x,\,\rho_i)}(x)\r|\\
&&\ls 1+(T^{*})^{\sharp}\chi_{F}(x)\ls1+\cm(T^{*}\chi_F)(x).
\end{eqnarray*}
Furthermore, by this, \eqref{3.13}, $E=\bigcup_{i=1}^N E_i$
(disjoint union) and $\mu(E)=\frac1t$,
we have that
\begin{eqnarray}\label{3.14}
\lf|\int_\cx\sz(x)\chi_F(x)\,d\mu(x)\r|&&
\le\dsum_{i=1}^N\lf|\int_\cx\chi_{E_i}(x)
\lf[T^{*}\chi_{F\setminus \overline B(x_i,\,2\rho_i)}\r](x)\,d\mu(x)\r|\\
&&\ls\dsum_{i=1}^N\int_\cx
\chi_{E_i}(x)[1+\cm(T^{*}\chi_F)(x)]\,d\mu(x)\nonumber\\
&&\sim\frac1t+\int_\cx\chi_E(x)\cm(T^{*}\chi_F)(x)\,d\mu(x).\nonumber
\end{eqnarray}
Since $T$ is bounded on $\lt$, by duality, we see that $T^\ast$ is also
bounded on $\lt$ and
$$\|T^{*}\|_{\lt\to\lt}=\|T\|_{\lt\to\lt}.$$
From this fact, Lemma \ref{l2.3}(i),
$\mu(F)=\frac1t=\mu(E)$ and the H\"older inequality, we further deduce that
\begin{eqnarray*}
\int_\cx\chi_E(x)\cm(T^{*}\chi_F)(x)\,d\mu(x)
&&\le\|\chi_E\|_\lt\|\cm(T^{*}\chi_F)\|_\lt\\
&&\le\|\chi_E\|_\lt\|\cm\|_{\lt\to\lt}\|T^{*}\|_{\lt\to\lt}\|\chi_F\|_\lt\\
&&=\frac{1}{t}\|\cm\|_{\lt\to\lt}\|T\|_{\lt\to\lt},
\end{eqnarray*}
which together with \eqref{3.14} gives that there exist $C_4 $ and $C_5$
satisfying \eqref{3.12}. Therefore the claim \eqref{3.12} holds.

Suppose that $\mu(\lf\{x\in\cx:\,|\sz(x)|>C_6 \r\})>2/t$. Then either
\begin{equation}\label{3.15}
\mu\lf(\lf\{x\in\cx :\,\sz(x)>C_6 \r\}\r)>\frac{1}{t}
\end{equation}
or
\begin{equation*}
\mu\lf(\lf\{x\in\cx :\,\sz(x)<-C_6 \r\}\r)>\frac{1}{t}.
\end{equation*}
Without loss of generality,  we may only consider \eqref{3.15} by
similarity. Pick some set $F\subseteq\cx$ with $\mu(F)=1/t$ such that
$\sz(x)>C_6 $ everywhere on $F$ (such $F$ exists because of Remark
\ref{r3.2}). Then apparently,
\begin{equation}\label{3.16}
\int_\cx\sz(x)\chi_F(x)\,d\mu(x)>\frac{C_6 }{t}.
\end{equation}
Thus, we get a contradiction by combining \eqref{3.12} with \eqref{3.16}, which
implies \eqref{3.11}, and hence completes the proof of Lemma 3.2.
\end{proof}

\begin{rem}\label{r3.3}\rm
(i) Theorem \ref{t3.1} also holds with finite linear combinations of
Dirac measures with arbitrary real coefficients. Indeed, every such
measure $\nu$ can be represented as $\nu=\nu_+-\nu_-$, where $\nu_+$ and
$\nu_-$ are finite linear combinations of Dirac measures with
positive coefficients and $\|\nu\|=\|\nu_+\|+\|\nu_-\|$. Therefore,
$\|T\nu\|_{\wlo}\le 2(C_1+C_2\|T\|_{\lt\to\lt})\|\nu\|$.

(ii) If we replace the assumption of
Theorem \ref{t3.1} that $T$ is bounded on $\lt$ by that $T$ is bounded
on $\lq$ for some $q\in(1, \infty)$, then via a slight modification of the
proof Theorem \ref{t3.1}, we have \eqref{3.6} with $\|T\|_{\lt\to\lt}$ replaced by
$\|T\|_{\lq\to\lq}$.
\end{rem}

\begin{proof}[Proof of Theorem \ref{t1.1}, Part I]
In this part, we show that (i) of Theorem \ref{t1.1}
implies (ii) and (iii) of Theorem \ref{t1.1} and
that (ii) of Theorem \ref{t1.1} implies (iii) of Theorem \ref{t1.1}.

We first assume that (i) holds and show that (ii) and (iii) hold.
By the Marcinkiewicz interpolation theorem and a duality argument, we
obtain (ii) via (iii). Therefore, we only need to prove (iii).
To this end, observe that for any $f\in\lo$, $f=f^+-f^-$,
where $f^+\equiv\max\{f, 0\}\ge 0$ and $f^-\equiv\max\{-f, 0\}\ge 0$.
Moreover, if let $C_b(\cx)$ be the {\it space of all
continuous functions with bounded support}, by \cite[Proposition 3.4]{h10}
and its proof, we see that for any $f\in\lo$ and $f\ge 0$,
there exist $\{f_j\}_{j\in\nn}\subseteq C_b(\cx)$ and $f_j\ge 0$ for
all $j\in\nn$ such that $\|f_j-f\|_\lo\to 0$ as $j\to\infty$.
By these observations combining with the linear property of $T$, we see that
to show (iii), it suffices to prove that \eqref{1.7} holds
for all $f\in C_b(\cx)$ and $f\ge 0$.

Let $t>0$, $G\equiv\{x\in\cx:\, f(x)>t\}$, $f^t\equiv
f\chi_G$ and $f_t\equiv f\chi_{\cx\setminus G}$. Then
$Tf=Tf^t+Tf_t$. Notice that
$$\int_\cx[f_t(x)]^2\,d\mu(x)\le t\int_\cx f_t(x)\,d\mu(x)\le
t\|f\|_{\lo}.$$
This and the boundedness of $T$ on $\lt$ yield that
$$\int_\cx|Tf_t(x)|^2\,d\mu(x)\le\|T\|_{\lt\to\lt}^2t\|f\|_\lo,$$
which implies that
\begin{equation}\label{3.17}
\mu\lf(\lf\{x\in\cx:\, |Tf_t(x)|>t\|T\|_{\lt\to\lt}\r\}\r)\le
\dfrac{\|f\|_{\lo}}{t}.
\end{equation}

We now estimate $Tf^t$. Since, by $f\in C_b(\cx)$,
$G$ is a bounded open set, by Lemma
\ref{l2.2}, there exists a sequence $\{B_i\}_i$ of
balls with finite overlap such that $G=\cup_i B_i$ and $2B_i\subseteq G$
for all $i$. Without loss of generality, we may assume
the cardinality of $\{B_i\}_i$ is just $\nn$.
Then the fact that $\{B_i\}_{i\in\nn}$ has the finite overlap
implies that
\begin{equation*}
  f^t=\sum_{i\in\nn} f\frac{\chi_{B_i}}{\sum_{j\in\nn}\chi_{B_j}}
  \equiv\sum_{i\in\nn}f_i.
\end{equation*}
Then it is easy to see that $f_i\ge 0$ for all $i\in\nn$.
For any $N\in \nn$ and $i\in\{1,\,2,\,\cdots, N\}$,
define $f^{(N)}\equiv\sum_{i=1}^{N}f_i$ and
$$\az_i\equiv\dint_\cx f_i(y)\,d\mu(y)=\dint_{B_i}f(y)\,d\mu(y).$$
Then $\az_i\ge 0$ for all $i\in\nn$. By $G=\cup_{i\in\nn} B_i$
and the finite overlap property of $\{B_i\}_{i\in\nn}$, we have
\begin{equation}\label{3.18}
\sum_{i=1}^{\fz}\az_i\le\sum_{i=1}^{\fz}
\int_{B_i}f(y)\,d\mu(y)\ls\int_G f(y)\,d\mu(y)\ls\|f\|_{L^1(\mu)}.
\end{equation}
 Pick $x_i\in B_i$ and define $\nu^{(N)}\equiv\sum_{i=1}^N\az_i\dz_{x_i}$.
 We obtain that $\|\nu^{(N)}\|=\sum_{i=1}^N\az_i$.
By \eqref{3.18}, the fact that $2B_i\subseteq G$ for all $i\in\nn$ and
\eqref{2.6}, there exists a positive constant $C_7$ such that
\begin{eqnarray}\label{3.19}
&&\dint_{\cx\setminus G}\lf|Tf^{(N)}(x)-T\nu^{(N)}(x)\r|\,d\mu(x)\\
&&\hs=\dint_{\cx\setminus G}\lf|T\lf(\dsum_{i=1}^{N}[f_i\,d\mu-
\az_i\dz_{x_i}]\r)(x)\r|\,d\mu(x)\nonumber\\
&&\hs\le\dsum_{i=1}^{N}\dint_{\cx\setminus 2B_i}\lf|T(f_i\,d\mu-
\az_i\dz_{x_i})(x)\r|\,d\mu(x)\noz\\
&&\hs\ls\dsum_{i=1}^{N}\az_i\le C_7\|f\|_{L^1(\mu)}.\noz
\end{eqnarray}

On the other hand, by Theorem \ref{t3.1}, we see that
\begin{equation*}
\mu\lf(\lf\{x\in\cx:\,\lf|T\nu^{(N)}(x)\r|>
(C_1+C_2\|T\|_{\lt\to\lt})t\r\}\r)\le\frac{1}{t}\lf\|\nu^{(N)}
\r\|\le\frac{1}{t}\|f\|_{\lo},
\end{equation*}
from which together with \eqref{3.19}, we deduce that
\begin{eqnarray*}
&&\mu\lf(\lf\{x\in\cx\setminus G:\,\lf|Tf^{(N)}(x)\r|>\lf(C_7+C_1+C_2
\|T\|_{\lt\to\lt}\r)t\r\}\r)\\
&&\hs\le\mu\lf(\lf\{x\in\cx\setminus G:\
\lf|Tf^{(N)}(x)-T\nu^{(N)}(x)\r|>C_7t\r\}\r)\\
&&\hs\hs+\mu\lf(\lf\{x\in\cx\setminus G:\,
\lf|T\nu^{(N)}(x)\r|>(C_1+C_2\|T\|_{\lt\to\lt})t\r\}\r)\le\frac{2}{t}\|f\|_{\lo}.
\end{eqnarray*}
This combined with the fact that $\mu(G)\le\|f\|_{\lo}/t$ implies
that
\begin{equation}\label{3.20}
\mu\lf(\lf\{x\in\cx:\,\lf|Tf^{(N)}(x)\r|>\lf(C_7+C_1+C_2\|T\|_{\lt\to\lt}\r)t\r\}\r)
\le\frac{3}{t}\|f\|_{\lo}.
\end{equation}

Observe that $f^{(N)}\to f^{t}$ in $\lt$ as $N\to\fz$. From the
$\lt$-boundedness of $T$, we then deduce
that $Tf^{(N)}\to Tf^{t}$ also in $\lt$ as $N\to\fz$. By this fact
and \eqref{3.20}, we have
$$\mu\lf(\lf\{x\in\cx:\, \lf|Tf^{t}(x)\r|>
\lf(C_7+C_1+C_2\|T\|_{\lt\to\lt}\r)t\r\}\r) \le\dfrac{3}{t}\|f\|_{L^1(\mu)},$$
from which
together with \eqref{3.17}, it follows that there exist positive
constants $C_8$ and $C_9$ such that
\begin{equation*}
\sup\limits_{t>0}t\,\mu(\{x\in\cx:\,|Tf(x)|>t\})\le\lf(C_8
+C_9\|T\|_{\lt\to\lt}\r)\|f\|_{\lo}.
\end{equation*}
This implies \eqref{1.7}, and hence finishes the proof of the implicity
${\rm (i)}\Rightarrow {\rm (iii)}$.

Now assume that (ii) holds. Then by Remark \ref{r3.3}(ii) and a similar
proof of ${\rm (i)}\Rightarrow {\rm (iii)}$, we see that (iii) holds.
We omit the details, which completes Part I of
the proof of Theorem \ref{t1.1}.
\end{proof}

\section{Proof of Theorem \ref{t1.1}, Part II}\label{s4}

\hskip\parindent
This section is devoted to the proof of (iii) $\Rightarrow$ (i)
of Theorem \ref{t1.1}. To do so, we first
establish the boundedness of $T^\sharp$ from $\lo$ to $\wlo$, which implies that
$\{T_r\}_{r\in(0, \fz)}$ is uniformly bounded from $\lo$ to $\wlo$. By restricting
$\mu$ to $\mu_M$,
where $\mu_M$ is the restriction of $\mu$ to a given
ball $\overline B(x_0, M)$ for some $x_0\in\cx$ and $M\in(0, \fz)$,
we will prove that for any $r\in(0, \fz)$ and $p\in(1, \fz)$, $T_r$
is bounded on $L^p(\mu_M)$. Then using a smooth
truncation argument, we will further show that $\{T_r\}_{r\in(0, \fz)}$ is
uniformly bounded from $\lt$ to $L^2(\mu_M)$ with the constant
independent of $M$. By letting $M\to\fz$, $\{T_r\}_{r\in(0, \fz)}$ is
uniformly bounded on $L^2(\mu)$. An argument involving the random
dyadic cubes from \cite{hm} will yield the desired conclusion.

\begin{thm}\label{t4.1}
Let $T$ be a Calder\'on-Zygmund operator with kernel satisfying
\eqref{1.5} and \eqref{1.6}, which is bounded
from $\lo$ to $\wlo$. Then there exists a positive constant $C$ such that
for any $f\in \lo$,
\begin{equation*}
\lf\|T^{\sharp}f\r\|_\wlo\le C\|f\|_\lo.
\end{equation*}
\end{thm}

\begin{proof}
Let $p\in(0, 1)$. By (i) and (ii) of Lemma \ref{l2.3}, we see that
$\cm$ is bounded from $\lo$ to $\wlo$, and
$\cm_p$ is bounded on $\wlo$. Then by the boundedness of
$T$ from $\lo$ to $\wlo$, to show Theorem \ref{t4.1},
we only need to prove that there exist positive
constants $C$ and $C(p)$ such that for any $f\in L^\fz_b(\mu)$ and $x\in\cx$,
\begin{equation*}
[T^{\sharp}f(x)]^p\ls[\cm _pTf(x)]^p+[\cm f(x)]^p.
\end{equation*}
Moreover, it suffices to prove that for any $r>0$, $f\in L_b^\fz(\mu)$ and $x\in\cx$,
\begin{equation}\label{4.1}
|T_rf(x)|^p\ls [\cm _pTf(x)]^p+[\cm f(x)]^p.
\end{equation}
To this end, for any $j\in\nn$, let $r_j\equiv5^jr$ and
$\mu_j\equiv\mu(\overline B(x,\,r_j))$ be as in the proof of Lemma \ref{l3.1}.
Again let $k$ be the
{\it smallest positive integer} such that
$\mu_{k+1}\le4C_{\lz}^6\mu_{k-1}$ and $R\equiv r_{k-1}=5^{k-1}r$.
Similarly to the proof of \eqref{3.3}, we see that
\begin{equation}\label{4.2}
|T_rf(x)-T_{5R}f(x)|\ls \cm f(x).
\end{equation}

Let $f_1\equiv f\chi_{\overline B(x,\, 5R)}$ and $f_2\equiv f-f_1$.
For any $u\in \overline B(x, R)$,
if $K$ is the kernel associated with $T$, then by \eqref{1.6} and
\eqref{1.3}, we see that
\begin{eqnarray*}
\lf|Tf_2(x)-Tf_2(u)\r|&&\le \dint_{d(x,\, y)>5R}
\lf|K(x, y)-K(u, y)\r||f(y)|\,d\mu(y)\\
&&\ls \dsum_{k=1}^\fz\lf[\frac{d(x, u)}{5^kR}\r]^\tau\dint_{\overline
B(x,\, 5^{k+1}R)}\frac{|f(y)|}{\lz(x, 5^kR)}\,d\mu(y)\ls\cm f(x).
\end{eqnarray*}
This, combined with \eqref{4.2}  and the fact that
\begin{equation*}
 Tf_2(x)=\dint_\cx K(x, y)f_2(y)\,d\mu(y)=T_{5R}f(x),
\end{equation*}
  implies that
\begin{eqnarray*}
|T_rf(x)|&&\le|T_r f(x)-T_{5R}f(x)|+|T_{5R}f(x)-Tf_2(u)|+|Tf_2(u)|\\
&&\ls \cm f(x)+|Tf(u)|+|Tf_1(u)|,
\end{eqnarray*}
from which and $p\in(0, 1)$, it further follows that for all $u\in \overline
B(x, R)$,
\begin{equation}\label{4.3}
|T_rf(x)|^p\ls \lf[\cm f(x)\r]^p+|Tf(u)|^p+|Tf_1(u)|^p.
\end{equation}

Since $T$ is bounded from $\lo$ to $\wlo$, by the Kolmogorov
inequality (see, for example, \cite[p.\,102]{d01}), we obtain that
\begin{equation}\label{4.4}
\dfrac1{\mu(\overline B(x, R))}\dint_{\overline B(x,\,R)}
|Tf_1(u)|^p\,d\mu(u)\ls \dfrac1{[\mu(\overline B(x, R))]^p}\lf[
\dint_{\overline B(x,\,R)}|f_1(u)|\,d\mu(u)\r]^p.
\end{equation}
Taking the average on the variable $u$ over $\overline B(x, R)$ on both sides of \eqref{4.3}, and
using \eqref{4.4}, the H\"older inequality and \eqref{3.2}, we see that
\begin{eqnarray*}
|T_rf(x)|^p&&\ls \lf[\cm f(x)\r]^p+
\lf[\cm_p(Tf)(x)\r]^p+\dfrac1{\mu(\overline B(x, R))}
\dint_{\overline B(x,\,R)}|Tf_1(u)|^p\,d\mu(u)\\
&&\ls  \lf[\cm f(x)\r]^p+ \lf[\cm_p(Tf)(x)\r]^p +\dfrac1{[\mu(\overline B(x, 25R))]^p}
\lf[\dint_{\overline B(x,\,5R)}|f(u)|\,d\mu(u)\r]^p\\
&&\ls\lf[\cm f(x)\r]^p+\lf[\cm_p(Tf)(x)\r]^p,
\end{eqnarray*}
which implies \eqref{4.1}, and hence completes the proof Theorem \ref{t4.1}.
\end{proof}

Let $x_0\in\cx$ and $M\in(0, \fz)$. We now obtain the boundedness of
the truncated operators
$\{T_r\}_{r\in(0, \fz)}$ on $L^p(\mu_M)$ for all
$p\in(1, \fz)$.
Notice that the set $X\setminus\overline{B}(x_0,M)$ has
$\mu_M$-measure zero by definition, and hence we may agree
that any $f\in L^p(\mu_M)$ satisfies $f|_{X\setminus\overline{B}(x_0,M)}\equiv 0$.
With this agreement, observe that
\begin{equation*}
  T_r f(x)=\int_{d(x,y)>r}K(x,y)f(y)\,d\mu(y)=\int_{d(x,y)>r}K(x,y)f(y)\,d\mu_M(y)
\end{equation*}
for $f\in L^p(\mu_M)$, so we may also replace $\mu$ by $\mu_M$ in the
formula of $T_r f$ when considering functions $f\in L^p(\mu_M)$.
Finally, observe that $\mu_M$ also satisfies the upper doubling
condition with the same dominating function $\lambda$, so that all
results shown for $\mu$ apply equally well to $\mu_M$, with constants
uniform with respect to $M$.

\begin{lem}\label{l4.1}
Let $p\in(1,\fz)$ and $r\in(0, \fz)$. Let $M\in(0, \fz)$
and $\mu_M$ be as above. Then there exists a positive constant
$\wz C$, depending on $M$ and $r$, such that for all $f\in L^p(\mu_M)$,
$$\|T_rf\|_{L^p(\mu_M)}\le \wz C\|f\|_{L^p(\mu_M)}.$$
\end{lem}

\begin{proof}
We first claim that there exists a positive constant $C$
such that for all $x\in \overline B(x_0, M)$,
\begin{equation}\label{4.5}
|T_rf(x)|\le C[\lz(x,r)]^{-1/p}\|f\|_{L^p(\mu_M)}.
\end{equation}
To this end, let $B_0\equiv B(x,r)$. Then \eqref{1.5} together with the H\"older
inequality gives that
\begin{equation}\label{4.6}
|T_rf(x)|\ls \lf[\int_{\cx\setminus B_0}\frac{d\mu(y)}
{[\lz(x,d(x,y))]^{p'}}\r]^\frac{1}{p'}\|f\|_{L^p(\mu_M)}.
\end{equation}

We prove the claim by inductively constructing an auxiliary
sequence $\{r_0, r_1, r_2,\ldots\}$ of radii
such that $r_0=r$ and $r_{i+1}$ is the smallest $2^k r_i$ with $k\in\nn$ satisfying
\begin{equation}\label{4.7}
\lz(x, 2^kr_i)>2 \lz(x, r_i),
\end{equation}
whenever such a $k$ exists.
We consider the following two cases.

{\it Case} (i) For each $i\in\zz_+$, there exists
$k\in\nn$ such that \eqref{4.7} holds.
In this case, $r_{i+1}$ will be the smallest $2^kr_i$
satisfying \eqref{4.7} for all $k\in\nn$,
and $\{B_i\}_{i\in\nn}\equiv \{B(x, r_i)\}_{i\in\nn}$.
Now by \eqref{1.3} and the fact that $2^i\lz(x, r)\le \lz(x, r_i)$
for all $i\in\zz_+$, we have that
\begin{eqnarray}\label{4.8}
\int_{\cx\setminus
B_0}\frac{d\mu(y)}{[\lz(x,d(x,y))]^{p'}}
&&\ls\dsum_{i=0}^{\fz}\frac{\mu(B_{i+1})}{[\lz(x,r_{i+1})]^{p'}}
\ls\sum_{i=0}^{\fz}\frac{1}{[\lz(x, r_{i+1})]^{p'-1}}\\
&&\ls\sum_{i=0}^{\fz}\frac{1}{[2^i\lz(x,r)]^{p'-1}}
\sim\frac{1}{[\lz(x, r)]^{p'-1}}\nonumber
\end{eqnarray}
and hence
$$\lf[\int_{\cx\setminus B_0}\frac{d\mu(y)}
{[\lz(x,d(x,y))]^{p'}}\r]^\frac{1}{p'}
\ls[\lz(x,r)]^{-\frac{1}{p}},$$
which combined with \eqref{4.6} implies \eqref{4.5}
and the claim holds in this case.

{\it Case} (ii) For some $i_0\in\mathbb{Z}_+$, \eqref{4.7} holds for
all $i< i_0$ but does not hold for  $i_0$. In this case, if
$i_0\in\nn$, we let $\{B_i\}_{i=1}^{i_0}$ as in Case (i),
$r_{i_0+1}\equiv\fz$ and $B_{i_0+1}\equiv\cx$; otherwise, if
$i_0=0$, we then let $r_1\equiv\fz$ and $B_1\equiv\cx$. Then we see that
$\lz(x,2^kr_{i_0})\le2\lz(x,r_{i_0})$ for all $k\in\nn$ and
$$\mu(\cx)\equiv\lim_{t\to\fz}\mu(B(x,t))\le
\lim_{t\to\fz}\lz(x,t)\equiv\lz(x,\fz)\le2\lz(x,r_{i_0}),$$
which, together with \eqref{1.3} and the fact that
$2^i\lz(x, r)\le \lz(x, r_i)$ for all $i\le i_0$, gives \eqref{4.8}
in this case, and the claim holds.

If $x\in\supp\mu_M=\overline B(x_0,M)$, then
$\supp\mu_M\subseteq B(x, 3M)$. By this and the definition
of $\supp\mu_M$, we get that
$$\mu_M(\cx)=\mu_M(B(x, 3M))\le\lz(x, 3M)\le C_\lz^{1+\log_2(3M/r)}\lz(x,r),$$
thus
$$\frac{1}{\lz(x,r)}\le\frac{C_\lz^{3+\log_2(M/r)}}{\mu_M(\cx)}.$$
By this fact, we obtain that
$$\int_\cx\frac{d\mu_M(\cx)}{\lz(x,r)}\le\frac{C_\lz^{3+\log_2(M/r)}}
{\mu_M(\cx)}\int_\cx\,d\mu_M(x)\le C_{\lz}^{3+\log_2(M/r)}.$$
From this and \eqref{4.5}, it follows that
\begin{align*}
  \|T_rf\|_{L^p(\mu_M)}
  &\ls \|f\|_{L^p(\mu_M)}\lf[\int_\cx\frac{d\mu_M(x)}{\lz(x,r)}\r]^\frac{1}{p} \\
  &\ls\|f\|_{L^p(\mu_M)}\lf[C_\lz^{3+\log_2(M/r)}\r]^\frac{1}{p}
   =\widetilde{C}(M,r)\|f\|_{L^p(\mu_M)}.
\end{align*}
This finishes the proof of Lemma \ref{l4.1}.
\end{proof}

We will need the following result which shows that two bounded Calder\'on--Zygmund
operators having the same kernel can at most differ by a multiplication operator.

\begin{prop}\label{p4.1}
Let $T$ and $\wz T$ be Calder\'{o}n-Zygmund operators which
have the same kernel satisfying
\eqref{1.5} and \eqref{1.6} and are both bounded from $L^p(\mu)$ to
$L^{p,\,\fz}(\mu)$ for some $p\in[1,\infty)$. Then
there exists $b\in L^{\infty}(\mu)$ such that for all $f\in L^p(\mu)$,
\begin{equation*}
  Tf-\wz Tf=bf\quad  {\rm{and}} \quad\Norm{b}{L^{\infty}(\mu)}\leq\Norm{T-\wz T}
  {L^p(\mu)\to L^{p,\infty}(\mu)}.
\end{equation*}
\end{prop}

The proof will rely on the following lemma.

\begin{lem}\label{l4.2}
For a suitable $\delta\in(0,1)$, there exists a sequence of countable
Borel partitions, $\{Q^k_{\alpha}\}_{\alpha\in\mathscr{A}_k}$, $k\in\zz$,
of $\cx$ with the following properties:
\begin{itemize}
  \item[\rm(i)] For some $x^k_{\alpha}\in\cx$ and constants $0<c_1<c_2<\infty$,
  $B(x^k_{\alpha},c_1\delta^k)\subseteq Q^k_{\alpha}\subseteq
  B(x^k_{\alpha},c_2\delta^k)$;
  \item[\rm(ii)] $\{Q^{k+1}_{\alpha}\}_{\alpha\in\mathscr{A}_{k+1}}$ is
  a refinement of $\{Q^k_{\alpha}\}_{\alpha\in\mathscr{A}_k}$.
\end{itemize}
Moreover, it may be arranged that
\begin{equation}\label{4.9}
\mu\lf(\bigcup_{\az,\,k}\partial Q^k_{\alpha}\r)=0,
\end{equation}
where for a set $Q$,  $\partial Q\equiv\{x\in\cx:\ d(x,Q)
=d(x, \cx\setminus Q)=0\}$ is the boundary.
\end{lem}

\begin{proof}
Let $\{Q^k_\az\}_{\az,\,k\in\zz}$ be the \emph{random dyadic cubes}
constructed in \cite{hm}, so in fact $Q^k_\az=Q^k_\az(\omega)$ where $\omega$ is a point of an underlying \emph{probability space}  $\boz$.
We use $\mathbb{P}$ to denote a \emph{probability measure} on $\Omega$ (as constructed in \cite{hm}), so that $\mathbb{P}(A)$ is \emph{probability of the event $A\subset\boz$}.
By the construction given in \cite{hm}, these sets automatically satisfy the other claims for all $\omega\in\Omega$, and it remains to show that we can choose $\omega\in\Omega$ so as to also satisfy \eqref{4.9}.

The ``side-length'' of $Q^k_{\alpha}$ is defined $\ell(Q^k_{\alpha})
\equiv\delta^k$, where $\delta\in(0,1)$ is a fixed parameter
entering the construction. For $\eps\in(0, \fz)$, let
$$\delta_{\eps}Q\equiv\{x:\ d(x,Q)\le\eps\ell(Q)\}
\bigcap\{x:\ d(x,\cx\setminus Q)\le\eps\ell(Q)\}.$$
It was shown in \cite[Lemma 10.1]{hm} that there exists an
$\eta>0$ such that for any fixed $x\in \cx$ and $k\in \zz$,
\begin{equation*}
  \prob\lf(x\in\bigcup_{\alpha} \delta_{\eps}{Q^k_{\alpha}}\r)
  \lesssim \eps^{\eta}.
\end{equation*}
In particular, by taking the limit as $\eps\to 0$, we obtain that
\begin{equation*}
  \prob\lf(x\in\bigcup_{\alpha} \partial Q^k_{\alpha}\r)=0.
\end{equation*}
Then it is possible to sum the zero probabilities over $k\in\Z$ to deduce
\begin{equation*}
  \prob\lf(x\in\bigcup_{k,\,\alpha} \partial Q^k_{\alpha}\r)=0.
\end{equation*}
Now we can compute (the integration variable of the
$\ud\prob$-integrals is $\omega\in\Omega$, the random
variable implicit in the random dyadic cubes
$Q^k_{\alpha}=Q^k_{\alpha}(\omega)$):
\begin{align*}
  \int_{\Omega} \mu\lf(\bigcup_{k,\,\alpha}
  \partial Q^k_{\alpha}\r)\ d\prob
  &=\int_{\Omega}\int_{\cx} 1_{\bigcup_{k,\,\alpha}
  \partial Q^k_{\alpha}}(x)\ d\mu(x)\ d\prob
  =\int_{\cx} \int_{\Omega}1_{\bigcup_{k,\,\alpha}
  \partial Q^k_{\alpha}}(x)\ud\prob\ d\mu(x)\\
  &=\int_{\cx} \prob\lf(x\in \bigcup_{k,\,\alpha}
  \partial Q^k_{\alpha}\r)\,d\mu(x)= 0.
\end{align*}
So, the integral of $\mu(\cup_{k,\,\alpha}
\partial Q^k_{\alpha}(\omega))\geq 0$ is zero.
This means that $\mu(\cup_{k,\,\alpha}
\partial Q^k_{\alpha}(\omega))= 0$ for $\prob$-almost every $\omega\in\Omega$.
Now we just fix one such $\omega$, and for this choice,
the boundaries of the corresponding dyadic cubes
$Q^k_{\alpha}=Q^k_{\alpha}(\omega)$
have $\mu$-measure zero.
This implies \eqref{4.9} and hence finishes the proof of Lemma \ref{l4.2}.
\end{proof}

\begin{proof}[Proof of Proposition~\ref{p4.1}]
Let $S\equiv T-\wz T$. Then
$S$ is bounded from $L^p(\mu)$ to $L^{p,\,\infty}(\mu)$ for
some $p\in[1, \fz)$ as in the proposition,
and it has kernel $0$.
We will prove that for all $M\in\nn$ and all
$f\in L^p(\mu)$ with $\supp f\subseteq B_M\equiv \overline B(x_0, M)$,
and $\mu$-almost every $x\in\cx$,
\begin{equation}\label{4.10}
S f(x)=f(x)S\lf(1_{B_M}\r)(x)\equiv f(x)b_M(x)
\end{equation}
and
\begin{equation}\label{4.11}
\|b_M\|_{L^\fz(\mu_M)}\le \|S\|_{L^p(\mu)\to L^{p,\infty}(\mu)},
\end{equation}
 where $\mu_M\equiv\mu|_{B_M}$.

Suppose for the moment that \eqref{4.9} and \eqref{4.10}
are already verified. If $M<M'$, then for all $f\in L^p(\mu)$
with $\supp f\subseteq B_M\subseteq B_{M'}$, we have $fb_M=Sf=fb_{M'}$
almost everywhere on $B_M$. Since this is true for all such $f$,
we must have $b_{M'}=b_M$ on $B_M$, and hence we can unambiguously
define $b(x)$ for all $x\in\cx$ by setting $b(x)\equiv b_M(x)$ for
$x\in B_M$. The uniform bound \eqref{4.10} implies that
$\|b\|_{L^{\infty}(\mu)}\leq\|S\|_{L^p(\mu)\to L^{p,\infty}(\mu)}$,
and we have $Sf = bf$ for all $f\in L^p(\mu)$ with bounded support.
Finally, by density this holds for all $f\in L^p(\mu)$. Thus,
proving \eqref{4.9} and \eqref{4.10} will prove the proposition,
and we turn to this task.

Now we prove \eqref{4.9}. Let us consider functions of the form
\begin{equation}\label{4.12}
  \sum_{\alpha} x^k_{\alpha} 1_{Q^k_{\alpha}\cap B_M},
\end{equation}
where $\{Q^k_{\alpha}\}_{\az,\,k}$ are the dyadic cubes with
zero-measure boundaries, as provided by Lemma \ref{l4.2}.
Since $(\cx, d)$ is geometrically doubling and $B_M$ is bounded, we
see that only finitely many $Q^k_{\alpha}$ intersect $B_M$, and hence the
sum in \eqref{4.12} may taken to be finite.

We claim that for $\mu$-almost every $x\in\cx$,
\begin{equation}\label{4.13}
  S\lf(1_{Q^k_{\alpha}\cap B_M}\r)(x)=1_{Q^k_{\alpha}\cap B_M}(x)
  \cdot S\lf(1_{B_M}\r)(x).
\end{equation}
Indeed, observe first that for $\mu$-almost every $x\in \cx$,
\begin{equation}\label{4.14}
  S\lf(1_{B_M}\r)(x)=S\lf(\sum_{\beta}1_{Q^k_{\beta}\cap B_M}\r)(x)
  =\sum_{\beta}S\lf(1_{Q^k_{\beta}\cap B_M}\r)(x).
\end{equation}
On the other hand, the assumption that $S$ has kernel $0$
means that for any $f\in L^\fz_b(\mu)$ and $\mu$-almost every $x\notin\supp f$,
\begin{equation*}
  Sf(x)=\int_\cx 0f(y)\ud\mu(y) = 0.
\end{equation*}
This gives that
\begin{eqnarray*}
  \supp \lf(S\lf(1_{Q^k_{\beta}\cap B_M}\r)\r)&&\subseteq\supp 1_{Q^k_{\beta}\cap B_M}
  =\overline{Q^k_{\beta}\bigcap B_M}\\
   &&\subseteq \overline{Q^k_{\beta}}\bigcup \overline B_M=\lf(Q^k_{\beta}\bigcap
   \overline B_M\r)\bigcup \lf(\partial Q^k_{\beta}\bigcap \overline B_M\r).
\end{eqnarray*}
Recall that $Q^k_{\alpha}$ and $Q^k_{\beta}$
are disjoint if $\alpha\neq\beta$,  which together with \eqref{4.11}
implies that almost every $x\in Q^k_{\alpha}\cap B_M$ is outside
$\supp (S(1_{Q^k_{\beta}\cap B_M}))$. Hence $S(1_{Q^k_{\beta}\cap B_M})(x)=0$
for $\mu$-almost every $x\in Q^k_{\alpha}\cap B_M$, and thus,
for $\mu$-almost every $x\in \cx$,
\begin{equation*}
  1_{Q^k_{\alpha}\cap B_M}(x)S\lf(1_{Q^k_{\beta}\cap B_M}\r)(x)
  =\delta_{\alpha\beta}1_{Q^k_{\alpha}\cap B_M}(x)
  S\lf(1_{Q^k_{\alpha}\cap B_M}\r)(x)
  =\delta_{\alpha\beta}S\lf(1_{Q^k_{\alpha}\cap B_M}\r)(x),
\end{equation*}
where $\dz_{\az\bz}\equiv1$ if $\az=\bz$ and $\dz_{\az\bz}\equiv0$ otherwise,
and the last equality follows from the fact that $1_{Q^k_{\alpha}\cap B_M}(x)=1$
for $\mu$-almost every $x\in\supp (S(1_{Q^k_{\alpha}\cap B_M}))$.
Multiplying \eqref{4.14} by
$1_{Q^k_{\alpha}\cap B_M}$ gives
\begin{equation*}
  1_{Q^k_{\alpha}\cap B_M}(x)S\lf(1_{B_M}\r)(x)=\sum_{\beta}
  1_{Q^k_{\alpha}\cap B_M}(x)S\lf(1_{Q^k_{\beta}\cap B_M}\r)(x)
  =S\lf(1_{Q^k_{\alpha}\cap B_M}\r)(x),
\end{equation*}
which is precisely \eqref{4.13}.

Now it is easy to complete the proof of \eqref{4.9}. For any $f$ of the form
\eqref{4.12}, it follows from \eqref{4.13} that
\begin{equation}\label{4.15}
  Sf=\sum_{\alpha}x^k_{\alpha}S\lf(1_{Q^k_{\alpha}\cap B_M}\r)=
  \sum_{\alpha}x^k_{\alpha}1_{Q^k_{\alpha}\cap B_M}S\lf(1_{B_M}\r)
  =fS\lf(1_{B_M}\r).
\end{equation}
 On the other hand, recall that martingale convergence implies that
for any $f\in\lo$,
\begin{equation*}
  \Exp_kf\equiv\sum_{\alpha} \ave{f}_{Q^k_{\alpha}} 1_{Q^k_{\alpha}} \to f
\end{equation*}
for $\mu$-almost every $x\in\cx$ and in $L^p(\mu)$ as $k\to\infty$.
If $f\in L^p(\mu)$ is general, apply \eqref{4.15} to
$\Exp_k f\cdot 1_{B_M}$.
Then as $k\to\infty$, we have $\Exp_k f\cdot 1_{B_M}\to f\cdot1_{B_M}$
in $L^p(\mu)$, hence $S(\Exp_k f\cdot1_{B_M})\to S(f\cdot 1_{B_M})$ in
$L^{p,\,\infty}(\mu)$,
and thus almost everywhere for a subsequence. Also,  by \eqref{4.15}, we obtain that
$$S\lf(\Exp_k f\cdot1_{B_M}\r)=\Exp_k f\cdot 1_{B_M}\cdot S\lf(1_{B_M}\r)
\to f\cdot 1_{B_M}\cdot S\lf(1_{B_M}\r)$$
for $\mu$-almost every $x\in\cx$. As a result,
for all $f\in L^p(\mu)$,
\begin{equation*}
  S(f\cdot 1_{B_M})= f\cdot1_{B_M}\cdot S\lf(1_{B_M}\r)
  \equiv  f\cdot 1_{B_M}\cdot b_M
\end{equation*}
where $b_M\equiv S\lf(1_{B_M}\r)\in L^{p,\,\infty}(\mu)$
since $1_{B_M}\in L^p(\mu)$. Thus, \eqref{4.9} holds for all $f\in L^p(\mu)$
with $\supp f\subseteq B_M$.

It remains to prove \eqref{4.10}. Let $\lz\in(0, \fz)$,
$f\equiv 1_{\{\abs{b_M}>\lambda\}\cap B_M}$ and
$$B\equiv \|S\|_{L^p(\mu)\to L^{p,\infty}(\mu)}.$$
Then $\|f\|_{L^p(\mu)}=[\mu(\{x\in\cx:\ \abs{b_M(x)}>\lambda\}\cap B_M)]^{1/p}$.
By this,
\eqref{4.9} and the boundedness of $S$ from
$L^p(\mu)$ to $L^{p,\infty}(\mu)$, we see that
\begin{eqnarray*}
  &&\lambda\lf[\mu(\{x\in\cx:\ \abs{b_M(x)}>\lambda\}\cap B_M)\r]^{1/p} \\
   &&\hs=\lambda\lf[\mu(\{x\in\cx:\ \abs{b_M(x)f(x)}>\lambda\})\r]^{1/p} \\
 &&\hs=\lambda\lf[\mu(\{x\in\cx:\ \abs{Sf(x)}>\lambda\})\r]^{1/p} \\
 &&\hs\le\|Sf\|_{L^{p,\infty}(\mu)}\leq B\|f\|_{L^p(\mu)} \\
 &&\hs= B\lf[\mu(\{x\in\cx:\ \abs{b_M(x)}>\lambda\}\cap B_M)\r]^{1/p}.
\end{eqnarray*}
This means that either $\mu(\{x\in\cx:\ \abs{b_M(x)}
>\lambda\}\cap B_M)=0$ or $\lambda\le
B$, which is the same as $\Norm{b_M}{L^{\infty}(\mu_M)}\le B$. This
implies \eqref{4.10}, and hence finishes the proof of Proposition \ref{p4.1}.
\end{proof}

From Proposition~\ref{p4.1}, we easily deduce the following consequence.

\begin{lem}\label{l4.3}
Let $T$ and $\wz T$ be Calder\'{o}n-Zygmund operators which
have the same kernel satisfying
\eqref{1.5} and \eqref{1.6} and are both bounded from $\lo$ to
$L^{1,\,\fz}(\mu)$. Assume that $\wz T$ is bounded on $\lt$.
Then $T$ is also bounded on $\lt$.
\end{lem}

\begin{proof}
By Proposition~\ref{p4.1}, we have $Tf=\wz Tf+bf$,
where $b\in L^{\infty}(\mu)$. Hence
\begin{equation*}
  \|Tf\|_{L^2(\mu)}\leq\|\wz Tf\|_{L^2(\mu)}+\|bf\|_{L^2(\mu)}
  \leq\big(\|\wz T\|_{L^2(\mu)\to L^2(\mu)}+\|b\|_{L^{\infty}(\mu)}
  \big)\|f\|_{L^2(\mu)},
\end{equation*}
which completes the proof of Lemma \ref{l4.3}.
\end{proof}

\begin{proof}[Proof of Theorem \ref{t1.1}, Part II]
In this part, we show that (iii) of Theorem \ref{t1.1} implies (i)
of Theorem \ref{t1.1}. Let $\mu_M\equiv \mu|_{\overline B(x_0, M)}$ be as before.
The assumption clearly implies that $T$ is bounded from
$L^1(\mu_M)$ to $L^{1,\,\infty}(\mu_M)$, with a norm bound independent of $M$.
We will then prove that $T$ is bounded on $L^2(\mu_M)$,
still with a bound independent of $M$. By the
density of boundedly supported $L_{\loc}^2(\mu)$-functions in $L^2(\mu)$
and the monotone convergence, this suffices to conclude the
proof of (iii) $\Rightarrow$ (i) of Theorem~\ref{t1.1}. Thus, from now on we work
with the measure $\mu_M$, recalling that it satisfies,
uniformly in $M$, the same assumptions as $\mu$, so that
everything shown for $\mu$ above equally well applies to $\mu_M$.

By Theorem \ref{t4.1}, we see that $T^\sharp$ is bounded from $L^1(\mu_M)$ to
$L^{1,\,\fz}(\mu_M)$, which implies that $\{T_r\}_{r\in(0, \fz)}$ is uniformly
bounded from $L^1(\mu_M)$ to $L^{1,\,\fz}(\mu_M)$, and the bound (denoted by
$N_1$) depends only on the norm of $T$ as the operator from
$L^1(\mu)$ to $L^{1,\fz}(\mu)$.

Let $p\in (1, \fz)$. It follows from Lemma \ref{l4.1} that for any
$r\in(0, \fz)$, $T_r$ is bounded on $L^p(\mu_M)$ with $p\in(1, \fz)$,
but with the norm a priori depending on $M$ and $r$.
We claim, however, that $\{T_r\}_{r\in(0, \fz)}$ is uniformly bounded
on $L^2(\mu_M)$. That is, if we denote the corresponding norm by
$N_p(r,M)$, then we have that there exists a  positive constant $C$
depending on $N_1$, but not on $r$ or $M$, such that
\begin{equation}\label{4.16}
N_2(r, M)\le C.
\end{equation}

To this end, we define for any $r\in(0, \fz)$ and $x\in\cx$,
$$T_r^\psi f(x)\equiv\int_\cx K(x,y)\psi\Big(\frac{d(x,y)}{r}
\Big)f(y)\,d\mu(y),$$
where $\psi$ is a smooth function on $(0, \fz)$ such
that $\supp\psi\subseteq[1/2, \fz)$, $\psi(t)\in[0,1]$
for all $t\in(0, \fz)$, and
$\psi(t)\equiv 1$ when $t\in[1,\fz)$, and $K$ is the kernel of $T$.
It follows, from the definition of $T^\psi_r$, \eqref{1.5}
and \eqref{1.3}, that for any $x\in\cx$,
\begin{eqnarray*}
\lf|T_rf(x)-T_r^\psi f(x)\r|
&&\le\int_{\overline B(x,\,r)\setminus B(x,\,r/2)}|K(x,y)||f(y)|\,d\mu(y)\\
&&\ls\dint_{\overline B(x,\,r)}\frac{|f(y)|}{\lz(x,r/2)}\,d\mu(y)\ls\cm f(x).
\end{eqnarray*}
This fact, together with Lemma \ref{l2.3}(i), implies
that the boundedness of $T_r$ on
$L^p(\mu_M)$ for
$p\in (1, \fz)$ or from $L^1(\mu_M)$ to $L^{1,\,\fz}(\mu_M)$
is equivalent to that of
$T_r^\psi$. Moreover, if $\{T_r\}_{r\in(0, \fz)}$ is uniformly bounded
on $L^p(\mu_M)$ or from $L^1(\mu_M)$ to $L^{1,\,\fz}(\mu_M)$, then so is
$\{T_r^\psi\}_{r\in(0, \fz)}$; and vice verse.

Now we denote by $\wz N_p(r, M)$ the norm of $T_r^\psi$
on $L^p(\mu_M)$ and by $\wz N_1$ the (finite) supremum
over $r$ and $M$ of the norms of $T_r^{\psi}$ from
$L^1(\mu_M)$ to $L^{1,\,\fz}(\mu_M)$.
Then to show \eqref{4.16}, we only need to
prove that
\begin{equation}\label{4.17}
\wz N_2(r, M)\le \wz C
\end{equation}
for some positive constant $\wz C$ independent of $r$ and $M$.

We now prove \eqref{4.17}. Observe that for each $r$, $T_r^\psi$ is
bounded on $L^2(\mu_M)$ and from $L^{1}(\mu_M)$
to $L^{1,\,\fz}(\mu_M)$. Then from the Marcinkiewicz interpolation
theorem, we deduce that $T_r^\psi$ is bounded on $L^{\frac43}(\mu_M)$
and $\wz N_{\frac43}(r,
M)\ls \wz N_1^\frac{1}{2}[\wz N_2(r, M)]^\frac{1}{2}$. By duality,
the right hand side gives also the bound for the norm of
$(T_r^\psi)^\ast$ on $L^4(\mu_M)$. Observe that
\begin{equation*}
(T^\psi_r)^\ast (g)(x)=\dint_\cx\overline{K(y,x)\psi
\lf(\frac{d(x, y)}{r}\r)}g(y)\,d\mu_M(y).
\end{equation*}
Then $(T_r^\psi)^\ast$ is also a Calder\'on-Zygmund operator.
Thus $(T_r^\psi)^\ast$ is bounded from
$L^1(\mu_M)$ to $L^{1,\,\fz}(\mu_M)$ and the norm is bounded by  $c\wz
N_1^\frac{1}{2}[\wz N_2(r, M)]^\frac{1}{2}+\wz c$ for some positive
constants $c$ and $\wz c$. Another application of the Marcinkiewicz
interpolation theorem yields that the norm of $(T_r^\psi)^\ast$ on
$L^{\frac43}(\mu_M)$ is also bounded by $c\wz
N_1^\frac{1}{2}[\wz N_2(r, M)]^\frac{1}{2}+\wz c$. By duality, we
further see that $\wz N_4(r, M)\le c\wz N_1^\frac{1}{2}[\wz N_2(r,
M)]^\frac{1}{2}+\wz c$. Using interpolation again, we have that $\wz
N_2(r, M)\le c\wz N_1^\frac{1}{2}[\wz N_2(r, M)]^\frac{1}{2}+\wz c$,
from which \eqref{4.17} follows. Thus, \eqref{4.16} holds and the
claim is true.

As a result of \eqref{4.16}, we see that $\{T_r\}_{r\in(0, \fz)}$ is uniformly
bounded on $L^2(\mu_M)$, with bounds also uniform in $M$. By letting $M\to\fz$, we
have that $\{T_r\}_{r\in(0, \fz)}$ is uniformly bounded on $\lt$.
Then there exists a weak limit $\tilde{T}$ bounded on $L^2(\mu)$
and some sequence $r_i\to 0$ as $i\to \fz$. That is,
for all $f\in L^2(\mu)$ and $g\in L^2(\mu)$,
\begin{equation*}
  \pair{g}{\tilde{T}f}=\lim_{r_i\to 0}\pair{g}{T_{r_i}f}.
\end{equation*}

By a standard argument (see, for example, \cite[Proposition
8.1.11]{g09}), it is easy to check that $\tilde{T}$ is a
Calder\'on--Zygmund operator with the same kernel $K(x,y)$ as $T$.
It follows, from ${\rm (i)}\Rightarrow{\rm (iii)}$ of
Theorem \ref{t1.1} for the operator $\tilde{T}$,
that $\tilde{T}$ is also bounded
from $L^1(\mu)$ to $L^{1,\,\infty}(\mu)$.
Applying Lemma \ref{l4.3}, we have that $T$ is also bounded
on $\lt$. This finishes the proof of (iii) $\Rightarrow$ (i)
of Theorem \ref{t1.1} and hence the proof of Theorem \ref{t1.1}.
\end{proof}

\section{Proof of Corollary \ref{c1.1}}\label{s5}

\hskip\parindent
As an application of Theorem \ref{t1.1}, we
prove Corollary \ref{c1.1} in this section.
We begin with an inequality for $T^\sharp$
on the elementary measures.

\begin{lem}\label{l5.1}
Let $p\in(0,1)$ and $T$ be a Calder\'on-Zygmund operator
with kernel satisfying \eqref{1.5} and \eqref{1.6},
which is bounded on $\lt$. Then there
exist positive constants $C$ and $C(p)$ such that
for all elementary measures
$\nu=\sum_i\alpha_i\delta_{x_i}$ and $x\in\supp\mu$,
\begin{equation}\label{5.1}
\lf[T^{\sharp}\nu(x)\r]^p\le C\lf[\cm_pT\nu(x)\r]^p+C(p)[\cm\nu(x)]^p.
\end{equation}
\end{lem}

\begin{proof}
As in Lemma \ref{l3.1}, let $r\in(0, \fz)$, $r_j\equiv5^jr$,
$\mu_j\equiv\mu(\overline B(x,\,r_j))$ for
$j\in\zz_+$, $k$ be the {\it smallest positive integer} such that
$\mu_{k+1}\le4C_{\lz}^6\mu_{k-1}$ and $R\equiv r_{k-1}=5^{k-1}r$.
Similarly to the proof of \eqref{3.3}, we have
\begin{equation}\label{5.2}
|T_r\nu(x)-T_{5R}\nu(x)|\ls\cm\nu(x).
\end{equation}
Now decompose the measure $\nu$ as $\nu=\nu_1+\nu_2$, where
$$\nu_1\equiv\dsum_{i:\, x_i\in
\overline B(x,\,5R)}\az_i\dz_{x_i}$$
and
$$\nu_2\equiv\dsum_{i:\, x_i\notin \overline
B(x,\,5R)}\az_i\dz_{x_i}.$$
Applying \eqref{2.4} to $T^\ast$,
we have that for any $\wz{x}\in\overline B(x,R)$,
\begin{eqnarray*}
|T_{5R}\nu(x)-T\nu_2(\wz{x})|&&
=\lf|\int_\cx K(x,y)\chi_{\cx\setminus
\overline B(x,\, 5R)}(y)\,d\nu(y)-T\nu_2(\wz{x})\r|\\
&&=\lf|\int_\cx K(x,y)\,d\nu_2(y)-T\nu_2(\wz{x})\r|\nonumber\\
&&=|T\nu_2(x)-T\nu_2(\wz{x})|=|\langle\dz_x,\,T\nu_2\rangle-
\langle\dz_{\wz{x}},\,T\nu_2\rangle|\nonumber\\
&&\le\int_\cx|T^{*}(\dz_x-\dz_{\wz{x}})(y)|\,d\nu_2(y)\nonumber\\
&&\le\int_{\cx\setminus \overline
B(x,\,5R)}|T^{*}(\dz_x-\dz_{\wz{x}})(y)|\,d\nu(y)\ls \cm\nu(x).\nonumber
\end{eqnarray*}
This implies that
\begin{equation}\label{5.3}
{\rm H}_1\equiv\frac{1}{\mu(\overline B(x,R))}
\int_{\overline B(x,\,R)}|T_{5R}\nu(x)-T\nu_2(\wz{x})|^p
\,d\mu(\wz{x})\ls[\cm \nu(x)]^p.
\end{equation}

On the other hand, write
\begin{eqnarray*}
{\rm H}_2&&\equiv\frac{1}{\mu(\overline B(x,R))} \int_{\overline
B(x,\,R)}|T\nu_2(\wz{x})-T\nu(\wz{x})|^p
\,d\mu(\wz{x})\\
&&=\frac{1}{\mu(\overline B(x,R))} \int_{\overline B(x,\,R)}|T\nu_1(\wz{x})|^p
\,d\mu(\wz{x})\nonumber\\
&&=\frac{1}{\mu(\overline B(x,R))} \int_0^{\fz}ps^{p-1}\mu\lf(\lf\{\wz{x}\in
\overline B(x,R):\, |T\nu_1(\wz{x})|>s\r\}\r)\,ds.\nonumber
\end{eqnarray*}
Since $T$ is bounded on $\lt$, by Theorem \ref{t3.1}, we have that
for every $s\in(0, \fz)$,
\begin{eqnarray}\label{5.4}
\mu\lf(\lf\{\wz{x}\in\overline B(x, R):\,|T\nu_1(\wz{x})|>s\r\}\r)\ls\min
\lf(\mu\lf(\overline B(x,
R)\r),\ \frac{\|\nu_1\|}{s}\r).
\end{eqnarray}
Observe that $\|\nu_1\|=\nu(\overline B(x, 5R))$.
This, together with \eqref{5.4}, the definition of $\cm\nu$ and \eqref{3.2},
gives that
\begin{eqnarray*}
\mu\lf(\lf\{\wz{x}\in\overline B(x, R):\,|T\nu_1(\wz{x})|>s\r\}\r)\,ds
&&\ls\mu\lf(\overline B(x, R)\r)\min\lf(1,\ \frac{1}{s}
\frac{\nu(\overline B(x, 5R))}{\mu(\overline B(x,R))}\r)\nonumber\\
&&\ls\mu\lf(\overline B(x, R)\r)\min\lf(1,\ \frac{1}{s}\cm\nu(x)\r),\nonumber
\end{eqnarray*}
which further implies that
\begin{eqnarray*}
{\rm H}_2&&\ls\int_0^\fz
ps^{p-1}\min\lf(1,\ \frac{1}{s}\cm\nu(x)\r)\,ds\nonumber\\
&&\sim\int_0^{\cm\nu(x)}ps^{p-1}\,ds\
+\int_{\cm\nu(x)}^{\fz}ps^{p-2}\cm\nu(x)\,ds\ls\lf[\cm\nu(x)\r]^p.\nonumber
\end{eqnarray*}
From this combined with \eqref{5.3}, we deduce that
\begin{eqnarray*}
&&\frac{1}{\mu(\overline B(x,R))} \int_{\overline
B(x,R)}\lf|T_{5R}\nu(x)-T\nu(\wz{x})\r|^p
\,d\mu(\wz{x})\ls {\rm H_1}+{\rm H}_2\ls[\cm\nu(x)]^p.\nonumber
\end{eqnarray*}
Using this and \eqref{5.2}, we see that
\begin{eqnarray*}
|T_r\nu(x)|^p&&=\frac{1}{\mu(\overline B(x,R))}
\int_{\overline B(x,R)}|T_r\nu(x)|^p\,d\mu(\wz{x})\\
&&\le\frac{1}{\mu(\overline B(x,R))} \int_{\overline B(x,R)}\lf[|T_r\nu(x)
-T_{5R}\nu(x)|^p\r.\nonumber\\
&&\hs\lf.+|T_{5R}\nu(x)-T\nu(\wz{x})|^p+|T\nu(\wz{x})|^p\r]
\,d\mu(\wz{x})\nonumber\\
&&\ls [\cm\nu(x)]^p+\frac{1}{\mu(\overline B(x,R))} \int_{\overline
B(x,R)}|T\nu(\wz{x})|^p\,d\mu(\wz{x})\nonumber\\
&&\ls[\cm\nu(x)]^p+\lf[\cm_pT\nu(x)\r]^p. \nonumber
\end{eqnarray*}
Taking the supremum over $r>0$, we see that \eqref{5.1} holds,
which completes the proof of Lemma \ref{l5.1}.
\end{proof}

As a result of Lemma \ref{l5.1}, by Theorem \ref{t3.1} and (i) and (ii) of
Lemma \ref{l2.3}, we have the following corollary.

\begin{prop}\label{p5.1}
Let $T$ be a Calder\'on-Zygmund operator with kernel satisfying
\eqref{1.5} and \eqref{1.6}, which is bounded on $\lt$. Then there exists a
positive constant $C$ such that for all elementary measures $\nu\in
\mathscr M(\cx)$,
$$\lf\|T^\sharp\nu\r\|_{L^{1,\,\fz}(\mu)}\le C\|\nu\|.$$
\end{prop}

\begin{proof}[Proof of Corollary \ref{c1.1}]
By Theorem \ref{t1.1}, Remark \ref{r3.1}, Lemma \ref{l2.3}(i) and
a density argument, we have (i). To prove (ii),
it suffices to prove \eqref{1.8}, since for any $f\in\lo$, if we
define $d\nu\equiv fd\mu$, then we see that $\nu\in\mathscr M(\cx)$
and \eqref{1.9} follows from \eqref{1.8}. Moreover, recall that for
any complex measure $\nu\in\mathscr M(\cx)$, $|\nu|(\cx)<\fz$; see,
for example, \cite[Theorem 6.4]{r}. Then by considering the Jordan
decompositions of real and imaginary parts of $\nu$, we only need to
prove \eqref{1.8} for any finite nonnegative measure.

To this end, assume that $\nu$ is a finite nonnegative measure and
fix $t>0$. We show that
$$\mu(\{x\in\cx:\, |T^{\sharp}\nu(x)|>t\})\ls
\dfrac{\|\nu\|}{t}.$$
Let $R>0$ and consider the truncated maximal
operator $T_R^{\sharp}\nu\equiv\sup_{r>R}|T_r\nu|$.
Since $T_R^{\sharp}\nu(x)$ increases to $T^{\sharp}\nu(x)$ pointwise
on $\cx$ as $R\to 0$, it suffices to show that there exists a positive
constant $C$ such that for every $R>0$,
\begin{equation}\label{5.5}
\mu\lf(\lf\{x\in\cx:\,\, \lf|T_R^{\sharp}\nu(x)\r|>t\r\}\r)\le \dfrac{C\|\nu\|}{t}.
\end{equation}

In what follows, we use $\mathbb{P}$ to denote a \emph{probability
measure} on a \emph{probability space}  $\boz$,
$\mathbb{P}(A)$ the \emph{probability of the event $A\subset\boz$},
$\mathbb{E}(\xi)$ the \emph{mathematical expectation of a random
variable $\xi\in L^1(\mathbb{P})$} and
$\mathbb{V}(\xi)\equiv\mathbb{E}[(\xi-\mathbb{E}\xi)^2]
=\mathbb{E}\xi^2-(\mathbb{E}\xi)^2$ the \emph{variance of} $\xi\in
L^2(\mathbb{P})$.

For each $N\in\nn$, consider the \emph{random elementary measure}
$\nu_N\equiv\frac{\|\nu\|}{N}\sum_{i=1}^N\dz_{x_i}$, where
the random points $\{x_i\}_{i=1}^N\subseteq\cx$ are independent and
$\mathbb P(\{x_i\in E\})=\nu(E)/\|\nu\|$ for every Borel set $E\subseteq \cx$.
This immediately implies that
\begin{equation*}
  \Exp f(x_i)=\frac{1}{\|\nu\|}\int_X f(z)d\nu(z)
\end{equation*}
for $f=1_E$ by definition, for simple functions $f$ by linearity,
and finally for all $f\in L^1(\nu)$ by approximation.
From this, we deduce that for every $x\in\cx$ and $r>R$,
\begin{equation}\label{5.6}
\mathbb{E}[(T_r\dz_{x_i})(x)]=\dfrac{1}{\|\nu\|}T_r\nu(x).
\end{equation}
Indeed,
\begin{align*}
  \|\nu\|\cdot\mathbb{E}[(T_r\dz_{x_i})(x)] &=\int_X (T_r\dz_{z})(x)d\nu(z) \\
  &=\int_X\int_{d(y,z)>r}K(x,y)d\delta_z(y)d\nu(z) \\
  &=\int_X  1_{d(x,z)>r} K(x,z)d\nu(z)=T_r\nu(x).
\end{align*}
Thus, \eqref{5.6} holds.

Fix some $x_0\in\cx$ and $M\in (R, \fz)$. On the other hand,
from \eqref{1.4} and \eqref{1.3}, we deduce that for
any $x\in \overline B(x_0, M)$,
\begin{equation*}
\lz(x_0, M)\ls\lz(x, M)\ls C_{\lz}^{1+\log_2(M/R)}\lz(x,R),
\end{equation*}
where $C_\lz$ is as in \eqref{1.3}. By this, the fact that $r>R$,
\eqref{5.6} and \eqref{1.5}, we have that
for any $x\in \overline B(x_0, M)$,
\begin{eqnarray}\label{5.7}
\mathbb{V}[T_r\dz_{x_i}(x)]&&\le\mathbb{E}\lf[|T_r\dz_{x_i}(x)|^2\r]
=\dint_\boz\lf[\dint_{\cx}K(x, y)
\,d\dz_{x_i}(y)\r]^2\,d\mathbb{P}\\
&&=\dint_\boz[K(x, x_i)]^2\chi_{\cx\setminus \overline B(x, \, r)}(x_i)\,d\mathbb{P}
\ls\dfrac1{[\lz(x,r)]^2}\ls\dfrac{C_\lz^{2[1+\log_2(M/R)]}}{[\lz(x_0, M)]^2}.\noz
\end{eqnarray}
Moreover, by \eqref{5.6}, we see that
\begin{equation}\label{5.8}
\mathbb{E}[(T_r\nu_N)(x)]=\dsum_{i=1}^N\dfrac{\|\nu\|}{N}
\mathbb{E}[(T_r\dz_{x_i})(x)]=T_r\nu(x).
\end{equation}
This, together with the Cauchy inequality and \eqref{5.7}, implies that
there exists a positive
constant $c$, independent of $x_0$, $M$, $r$, $R$ and $N$, such that
\begin{eqnarray*}
\mathbb{V}[T_r\nu_N(x)]&&=\dfrac{\|\nu\|^2}{N^2}\mathbb{V}
\lf[\sum_{i=1}^NT_r\dz_{x_i}(x)\r]
\le \dfrac{\|\nu\|^2}{N}\sum_{i=1}^N\mathbb{V}\lf[T_r\dz_{x_i}(x)\r]
\le c\dfrac{\|\nu\|^2}{N}\dfrac{C_\lz^{2[1+\log_2(M/R)]}}{[\lz(x_0, M)]^2}.
\end{eqnarray*}

Fix a  number $\gz\in(0,\fz)$ small enough. From the fact above,
the Chebyshev inequality and \eqref{5.8}, we deduce that for every point
$x\in\overline B(x_0, M)$ such that $|T_r\nu(x)|>t$,
\begin{eqnarray*}
\mathbb{P}(\{|T_r\nu_N(x)|\le(1-\gz)t\})
&&\le\mathbb{P}(\{|T_r\nu_N(x)-T_r\nu(x)|>\gz t\})\\
&&\le\dfrac{\mathbb{V}(T_r\nu_N)(x)}{\gz^2t^2}
\le c\dfrac{1}{\gz^2t^2} \dfrac{\|\nu\|^2}{N}
\dfrac{C_\lz^{2[1+\log_2(M/R)]}}{[\lz(x_0, M)]^2}\le\gz,
\end{eqnarray*}
provided $N\ge c\frac{\|\nu\|^2}{\gz^3t^2}
\frac{C_\lz^{2[1+\log_2(M/R)]}}{[\lz(x_0, M)]^2}$.
Since $r>R$ is arbitrary, we infer that for
each $x\in\cx$ satisfying
$T_R^\sharp\nu(x)>t$,
$$\mathbb{P}\lf(\lf\{T^\sharp_R\nu_N(x)\le(1-\gz)t\r\}\r)\le\gz.$$

Let $E$ be any given Borel set with $\mu(E)<\fz$ such that
$T_R^\sharp\nu(x)>t$ for every $x\in E$. Then
\begin{eqnarray*}
\mathbb{E}\lf(\mu\lf(\lf\{x\in E:\, T^\sharp_R\nu_N(x)\le(1-\gz)t\r\}\r)\r)
&&=\int_E\mathbb{P}\lf(\lf\{T^\sharp_R\nu_N(x)\le(1-\gz)t\r\}\r)\,d\mu(x)\\
&&\le\gz\mu(E).
\end{eqnarray*}
Thus there exists at least one choice of points $\{x_i\}_{i=1}^N$
such that $\mu(\{x\in E:\,T^\sharp_R\nu_N(x)\le(1-\gz)t\})
\le\gz\mu(E)$, and therefore,
$\mu(\{x\in E:\, T^\sharp_R\nu_N(x)>(1-\gz)t\})\ge(1-\gz)\mu(E)$.
From this together with Proposition \ref{p5.1}, it follows that
\begin{eqnarray*}
\mu(E)&&\le\dfrac{1}{1-\gz}\mu\lf(\lf\{x\in E:\,T^\sharp_R\nu_N(x)>(1-\gz)t\r\}\r)\\
&&\le\dfrac{1}{(1-\gz)^2t}\lf\|T^\sharp_R\nu_N\r\|_{L^{1,\,\fz}(\mu)}
\ls\dfrac{1}{(1-\gz)^2t}\|\nu_N\| \ls\dfrac{1}{(1-\gz)^2t}\|\nu\|.
\end{eqnarray*}

Since $\gz>0$ is arbitrary, we obtain that
$\mu(E)\ls\dfrac{\|\nu\|}{t}$. As $E$ is an arbitrary subset
of finite measure of the set of the points $x\in\cx$ for which
$T_R^\sharp\nu(x)>t$, we obtain \eqref{5.5}, which completes the
proof of Corollary \ref{c1.1}.
\end{proof}

\begin{rem}\label{r5.1}\rm
If we replace the assumption of Corollary \ref{c1.1} that $T$
is bounded on $\lt$ by that $T$ is bounded on
$\lq$ for some $q\in(1, \fz)$, then Corollary \ref{c1.1} still holds.
\end{rem}

\bigskip

\noindent Tuomas Hyt\"onen

\medskip

\noindent Department of Mathematics and Statistics, University of
Helsinki, Gustaf H\"allstr\"omin Katu 2B, Fi-00014 Helsinki, Finland

\medskip

\noindent{\it E-mail address}: \texttt{tuomas.hytonen@helsinki.fi}

\medskip

\noindent Suile Liu and Dachun Yang (Corresponding author)

\medskip

\noindent School of Mathematical Sciences,
 Beijing Normal University, Laboratory of Mathematics and Complex systems,
Ministry of Education, Beijing 100875, People's Republic of China

\medskip

\noindent{\it E-mail addresses}: \texttt{slliu@mail.bnu.edu.cn} (S. Liu)

                                \hspace{2.58cm}\texttt{dcyang@bnu.edu.cn} (D. Yang)

\medskip

\noindent Dongyong Yang

\medskip

\noindent School of Mathematical Sciences,
 Xiamen University, Xiamen 361005, People's Republic of China

\medskip

\noindent{\it E-mail address}: \texttt{dyyang@xmu.edu.cn}
\end{document}